\documentclass[preprint,12pt,a4paper]{elsarticle}
\usepackage[hidelinks]{hyperref}
\usepackage{amsmath}
\def\tsc#1{\csdef{#1}{\textsc{\lowercase{#1}}\xspace}}
\tsc{WGM}
\tsc{QE}
\usepackage{graphicx}
\usepackage{enumitem}
\usepackage{amssymb}
\DeclareMathAlphabet\mathbfcal{OMS}{cmsy}{b}{n}

\def\Emb{\mathop{\mathrm{Emb}}\nolimits}

\def\K{{\mathcal K}}

\def\str#1{\mathrm {#1}}

\def\rel#1#2{R_{\mathbf{#1}}^{#2}}

\def\Alphabet{\Sigma}
\def\VSpace#1#2#3{\left[#1\right]^*\!{#2\choose #3}}
\def\Space#1#2#3{\left[#1\right]{#2\choose #3}}

\newtheorem{theorem}{Theorem}
\newtheorem{lemma}[theorem]{Lemma}
\newtheorem{prop}[theorem]{Proposition}
\newtheorem{corollary}[theorem]{Corollary}
\newtheorem{observation}[theorem]{Observation}
\newtheorem{claim}[theorem]{Claim}
\newdefinition{remark}{Remark}
\newdefinition{example}{Example}
\newdefinition{definition}{Definition}
\newproof{proof}{Proof}

\def\str#1{\mathbf {#1}}
\def\Fraisse{Fra\"{\i}ss\' e}

\begin{document}
\let\WriteBookmarks\relax
\def\floatpagepagefraction{1}
\def\textpagefraction{.001}

% Short title
%\shorttitle{Big Ramsey degrees using parameter spaces}    

% Short author
%\shortauthors{Jan Hubička}  

% Main title of the paper
%\title [mode = title]{Big Ramsey degrees using parameter spaces}  
\title {Big Ramsey degrees using parameter spaces}

\author[1]{Jan Hubička}%[orcid=0000-0001-8704-0803]

% Corresponding author indication
%\cormark[1]

% Footnote of the first author
%\fnmark[1]

% Email id of the first author
\ead{hubicka@kam.mff.cuni.cz}

% URL of the first author
\ead[url]{http:/www.ucw.cz/~hubicka}

% Credit authorship
% eg: \credit{Conceptualization of this study, Methodology, Software}
%\credit{}

% Address/affiliation
\affiliation[1]{organization={Department of Applied Mathematics, Charles University},
            addressline={Malostranské náměstí 25}, 
            city={Praha 1},
%          citysep={}, % Uncomment if no comma needed between city and postcode
            postcode={118 00}, 
            %state={},
            country={Czech Republic}}
% Corresponding author text
\cortext[1]{Supported  by  project 21--10775S  of  the  Czech  Science Foundation (GA\v CR) and in later stages by ERC Synergy grant DYNASNET 810115}

% For a title note without a number/mark
%\nonumnote{}

\begin{abstract}
We show that the universal homogeneous partial order has finite big Ramsey degrees and discuss several corollaries. Our proof relies on parameter spaces and the Carlson--Simp\-son theorem rather than on (a strengthening of) the Halpern--L\"auchli theorem and the Milliken tree theorem, which are typically used to bound big Ramsey degrees in the existing literature (originating from the work of Laver and Milliken).

This new technique has many additional applications. 
We show that the homogeneous universal triangle-free graph has finite big Ramsey degrees, providing a short proof of a recent result by Dobrinen.
Moreover,
generalizing an indivisibility (vertex partition) result of Nguyen van Th\'e and Sauer, we give an upper bound on big Ramsey degrees of metric spaces with finitely
many distances.
This leads to a new combinatorial argument for the oscillation stability of the Urysohn Sphere.
\end{abstract}
%\subjclass[2020]{Primary 05D10, 06A07; Secondary 05C05, 05C55}
\maketitle
\section{Introduction}
We consider graphs, partial orders, (vertex)-ordered graphs, and partial orders with linear extensions as special cases of model-theoretic relational
structures (defined in Section~\ref{sec:preliminaries}).  Given structures $\str{A}$ and
$\str{B}$, we denote by $\Emb(\str{A}, \str{B})$ the set of all embeddings from
$\str{A}$ to $\str{B}$. We write $\str{C}\longrightarrow
(\str{B})^\str{A}_{r,l}$ to denote the following statement:
\begin{quote}
 For every colouring
$\chi$ of $\Emb(\str{A},\str{C})$ with $r$ colours, there exists an embedding
$f\colon \str{B}\to\str{C}$ such that $\chi$ does not attain more than $l$ values on
$\Emb(\str{A},f(\str{B}))$.
\end{quote}
  For a
countably infinite structure~$\str{B}$ and its finite sub\-struc\-ture~$\str{A}$,
the \emph{big Ramsey degree} of $\str{A}$ in $\str{B}$ is the least number
$L\in \mathbb \omega\cup \{\omega\}$ such that $\str{B}\longrightarrow
(\str{B})^\str{A}_{r,L}$ for every $r\in \mathbb \omega$; see~\cite{Kechris2005}.  
A countably
infinite structure $\str{B}$ has \emph{finite big Ramsey degrees} if the big
Ramsey degree of $\str{A}$ in $\str{B}$ is finite for every finite substructure $\str{A}$ of $\str{B}$.

A countable structure $\str{A}$ is called \textit{(ultra)homogeneous} if every
isomorphism between finite substructures extends to an automorphism of
$\str{A}$.  It is well known that there is, up to isomorphism, a unique
homogeneous partial order $\str{P}$ with the property that every countable
partial order has an embedding to $\str{P}$. We
call $\str{P}$ the \emph{universal homogeneous partial order}.  Similarly, there
is an up to isomorphism unique homogeneous triangle-free graph $\str{H}$ (called the \emph{universal homogeneous triangle-free graph}, sometimes also \emph{triangle-free Henson graph}) such that
every countable triangle-free graph embeds to $\str{H}$.
(See e.g.~\cite{Macpherson2011} for more background on homogeneous structures.)

Our main result is the following.
\begin{theorem}
\label{thm:posets}
The universal homogeneous partial order has finite big Ramsey degrees.
\end{theorem}
Until recently, only a few examples of structures with finite big
Ramsey degrees have been known.
As we show in Section~\ref{sec:applications} the universal homogeneous partial order represents an important new example of a structure in which many 
known examples (and some new) can be interpreted and thus follow as a direct consequence.
\medskip

The study of big Ramsey degrees originates in the work of Laver who, in 1969,
showed that the big Ramsey degrees of the order of rationals are finite~\cite[page~73]{devlin1979}, see also~\cite{erdos1974unsolved,laver1984products}. In his
argument, he re-invented the Halpern--L{\"a}uchli theorem~\cite{Halpern1966}. His
technique was later formulated more generally using the Milliken tree
theorem~\cite{Milliken1979} and the notion of envelopes and embedding types~\cite[Chapter~6]{todorcevic2010introduction}.  Most existing results in the area continue to
use the Milliken tree theorem as the primary proof technique. In particular, Devlin in
1979~\cite{devlin1979} refined Laver's argument thereby giving a precise
characterisation of the big Ramsey degrees of the order of rationals.
In 2005, this result was revived in the context of the Kechris--Pestov-Todorcevic correspondence~\cite{Kechris2005}.
In 2006,
Sauer~\cite{Sauer2006}, and Laflamme, Sauer, and Vuksanovic~\cite{Laflamme2006}
characterised big Ramsey degrees of the Rado graph (with precise counts
given by Larson~\cite{larson2008counting}). This was further generalised in
several follow-up papers~\cite{laflamme2010partition,dobrinen2016rainbow}.

Our proof of Theorem~\ref{thm:posets}, for the first time in the area, uses
spaces described by parameter words.  This leads to a finer control over the
sub-trees compared to the aforementioned constructions. Our main Ramsey tool, formulated as Theorem~\ref{thm:multCS}, is an
infinitary extension of the Graham--Rothschild theorem~\cite{Graham1971} and is
a direct consequence of the Carlson--Simp\-son theorem~\cite{carlson1984}.
While the connections of the Carlson--Simp\-son theorem, Halpern--L{\"a}uchli
theorem for trees with bounded branching and the Milliken
tree theorem are well known~\cite{carlson1984,dodos2016,todorcevic2010introduction}, so far the additional invariants
parameter spaces can preserve have not been applied in this context.

The proof technique presented in this paper is flexible and can be used to obtain additional finite big Ramsey degrees results
for restricted structures (that is, structures omitting given substructures or satisfying certain axioms).
To demonstrate this, we give a new short proof of the following recent result of
Dobrinen~\cite{dobrinen2017universal}:
\begin{theorem}[Dobrinen 2020~\cite{dobrinen2017universal}]
\label{thm:trianglefree}
The universal homogeneous triangle-free graph has finite big Ramsey degrees.
\end{theorem}
Both results have well-known finitary counterparts.
Given a class $\mathcal K$ of structures, the \emph{(small) Ramsey degree} of
$\str{A}$ in $\mathcal K$ is the least $l\in \mathbb N\cup\{\omega\}$ such that
for every $\str{B}\in \mathcal K$ and $r\in \mathbb N$ there exists $\str{C}\in
\mathcal K$ such that $\str{C}\longrightarrow (\str{B})^\str{A}_{r,l}$.
A class $\mathcal K$ of finite structures is \emph{Ramsey} (or has the \emph{Ramsey property}) if the
small Ramsey degree of every $\str{A}\in \mathcal K$ is one.
 The Ramsey property
for finite partial orders with linear extensions was announced by Ne\v set\v
ril and R\"odl in 1984~\cite{Nevsetvril1984}.
	One year later, using a different method,  Paoli, Trotter, and Walker proved a weaker
form of this result~\cite[Lemma~15 and Theorem~16]{Trotter1985}\footnote{This proof is sometimes considered faulty   since the paper states as Theorem 2 an infinite form of the product Ramsey theorem, which is known to hold only in its finite form. However, when Theorem 2 is applied to prove Lemma~15 and Theorem~16, the following remark saves the day: ``To simplify the presentation of an argument we take $\underset{\sim}{Z}$ to be an infinite poset. Of course, we can actually choose $\underset{\sim}{Z}$ as $\underline{p}^k$ where $p$ is a sufficiently large integer.''}. The basic idea of
the proof, based on a combination of the product Ramsey theorem and the dual
Ramey theorem, was adapted by Fouch{\'e}~\cite{Fouche1997} to determine small
Ramsey degrees of partial orders.  This result directly implies
that the class of all partial orders with linear extension is Ramsey as shown by Soki\'c \cite[Theorem
7(6)]{sokic2012ramsey}.  A self-contained presentation of this strategy was given
by Solecki and Zhao~\cite{solecki2017ramsey}, generalizing
the result to multiple linear extensions.
A different approach, based on the Graham--Rothschild
theorem alone, was found by Ma{\v{s}}ulovi{\'c} in 2018~\cite{masulovic2016pre} and later generalized to multiple partial orders and linear extensions jointly with Dragani{\'c}~\cite{draganic2019ramsey}.
In the same year, the original proof
using partite construction  was published by Ne\v set\v ril and R\"odl~\cite{nevsetvril2018ramsey}, see also recent survey~\cite{HubickaKonecnySurvey}.
In Section~\ref{superpose}, we present a method for obtaining big Ramsey equivalents of generalizations given by Solecki and Zhao, and Dragani{\'c} and Ma{\v{s}}ulovi{\'c}.

While there is a general framework which can be used to show that a given class
$\K$ is Ramsey~\cite{Hubicka2016}, the situation is very different in the context of
big Ramsey degrees as, despite the recent rapid progress, only relatively few types of structures
have big Ramsey degrees of their \Fraisse{} limits understood.
The main difference is the lack of an infinite variant of the (Ne\v set\v
ril and R\"odl's) partite construction~\cite{Nevsetvril1989} (see~\cite{nevsetvril2018ramsey} for its adaptation to partial orders) which has proved to be
a very versatile tool in the structural Ramsey theory.

For several decades, it was not clear how to generalize Laver's 
proof to (countable) restricted structures or structures in languages with relations of arity three or more.  
Dobrinen's recent proof of Theorem~\ref{thm:trianglefree} ignited a significant burst of progress.
Her proof uses a new method of bounding big Ramsey degrees
inspired by Harrington's proof of the Halpern--Läuchli theorem, which uses
techniques from set-theoretic forcing and the Erd\H os--Rado theorem.  The main pigeonhole argument is a technically challenging structured tree theorem, where the tree is built
using a particular enumeration of the graph $\str{H}$ in which certain tree levels are coding
(and contain vertices of the graph being represented) while others are
branching.  This method was later generalized
to (non-oriented) Henson graphs~\cite{dobrinen2019ramsey}. Recently, Zucker
simplified it and further generalized to finitely constrained free amalgamation
classes of structures in binary languages~\cite{zucker2020}. Zucker's proof is still
based on a structured pigeonhole proved by forcing techniques, but it greatly
simplifies the trees by eliminating distinction between coding and branching levels.
While this simplification gives larger upper bounds than one given by Dobrinen,
recently this technique has been refined to characterise degrees precisely. See \cite{Balko2021exact} and Section~\ref{sec:optimality}.

Bounds on big Ramsey degrees of
unrestricted structures with arities greater than 2 were announced in~\cite{Hubickabigramsey} with a proof based
on the vector (or product) form Milliken tree theorem~\cite{Hubicka2020uniform}. This technique was further generalized
to infinite languages~\cite{braunfeld2023big}. We believe that this represents the strongest possible big Ramsey results
based directly on unmodified form of the Milliken tree theorem.
This paper is motivated by the opposite direction and builds on proof techniques used for giving bounds on small Ramsey degrees
where the Graham--Rothschild Theorem is a common tool. 

We shall also remark that Theorem~\ref{thm:multCS} has a direct proof\footnote{By ``direct'' we mean that the proof is elementary
and does not use set-theoretic forcing, ultrafilters or topological dynamics.} based
on a combinatorial forcing argument (see the proof of the ``key lemma'' in~\cite{carlson1984} or Theorem 2 of \cite{Karagiannis2013}).
Consequently, we obtain the first direct (and simple) proof of Theorem~\ref{thm:trianglefree}. 

\medskip
The paper is organised as follows. In Section~\ref{sec:preliminaries} we
introduce parameter spaces.  In Section~\ref{sec:envelopes} we introduce
the corresponding notion of envelopes and embedding types.  In Section~\ref{sec:bigdegrees}
we prove the main results of this paper. In Section~\ref{sec:smallclasses} we show
that the construction is tight for determining small Ramsey degrees and thus
give a new proof of a special case of the Ne\v set\v ril--R\"odl
theorem~\cite{Nevsetvril1977}. This makes the connection between finitary and infinitary structural Ramsey results more explicit. In Section~\ref{sec:applications} we discuss several corollaries.
In Section~\ref{sec:concluding} we briefly outline ongoing work and further directions
to generalize techniques of this paper.
\section{Preliminaries}\label{sec:preliminaries}
We use the standard model-theoretic notion of relational structures.
Let $L$ be a language with relation symbols $\rel{}{}\in L$ each having its {\em arity}.
An \emph{$L$-structure} $\str{A}$ on $A$ is a structure with {\em vertex set} $A$ and relations $\rel{A}{}\subseteq A^r$ for every symbol $\rel{}{}\in L$ of arity $r$.  If the set $A$ is finite, we call $\str A$ a \emph{finite structure}. 
We will always use bold letters $\str{A},\str{B},\ldots$ to denote structures and $A,B,\ldots$ to denote their corresponding underlying sets.
We consider only structures with finitely many or countably infinitely many vertices.

Given two $L$-structures $\str{A}$ and $\str{B}$, a function $f\colon A\to B$ is an \emph{embedding} $f\colon \str{A}\to\str{B}$ if it is
injective and for every $\rel{}{}\in L$ of arity $r$ we have that $$(v_1,v_2,\ldots, v_r)\in \rel{A}{}\iff (f(v_1),f(v_2),\ldots, f(v_r))\in \rel{B}{}.$$ 
We say that $\str{A}$ and $\str{B}$ are \emph{isomorphic} if there is an embedding $f\colon \str{A}\to \str{B}$ that is onto.

As usual in the structural Ramsey theory, given an embedding $f\colon \str{A}\to\str{B}$ we will call the image of $\str{A}$ in $\str{B}$ (denoted by $f(\str{A})$) a \emph{copy} of $\str{A}$ in $\str{B}$.
A structure $\str{A}$ is \emph{rigid} if the only automorphism of $\str{A}$ (that is, isomorphism $\str A\to \str A$) is the
identity. 

\subsection{Parameter words and spaces}
Given a finite alphabet $\Alphabet$ and $k\in \omega\cup \{\omega\}$, a \emph{$k$-parameter word} is a (possibly infinite) string $W$ in
alphabet $\Alphabet\cup \{\lambda_i\colon 0\leq i<k\}$ containing each
of $\lambda_i$, $0\leq i < k$, such that for every $1\leq j < k$, the first
occurrence of $\lambda_j$ appears after the first occurrence of $\lambda_{j-1}$. 
Given a parameter word $W$, we denote by $|W|$ its \emph{length} and for every $0\leq j < |W|$ by $W_j$ the letter (or parameter) on index $j$. (Note that the first letter of $W$ has index $0$).
A $0$-parameter word is simply a \emph{word}. We will generally denote words by lowercase letters and parameter words by uppercase letters.

Let $W$ be an $n$-parameter word and let $U$ be a parameter word of length $k\leq n$ (where $k,n\in \omega\cup\{\omega\}$). Then we denote by
$W(U)$ the parameter word created by \emph{substituting} $U$ to $W$. More precisely, this is a parameter word created from $W$ by replacing each occurrence of $\lambda_i$, $0\leq i < k$, by $U_i$ and truncating it just
before the first occurrence of $\lambda_k$ (in $W$).
Given an $n$-parameter word $W$ and set $S$ of parameter words of length at most $n$, we denote by $W(S)$ the set $\{W(U)\colon U\in S\}$.

Given $k\leq n\in \omega\cup \{\omega\}$ we denote by $\Space{\Alphabet}{n}{k}$ the set of all $k$-parameter words of
length $n$. If $k$ is finite we also denote by $$\VSpace{\Alphabet}{n}{k}=\bigcup_{i\leq n,i\in \omega}\: \Space{\Alphabet}{i}{k}$$ the set of all finite $k$-parameter words
of length at most $n$.
For brevity we denote by  $\Alphabet^*$  the set $\VSpace{\Alphabet}{\omega}{0}$ of all words on the alphabet $\Alphabet$ with finite length and no parameters.
Given an $n$-parameter word $W$ and integer $k<n$, we call $W(\VSpace{\Alphabet}{n}{k})$ the
\emph{$k$-dimensional subspace} described by $W$.
We will denote by $\emptyset$ the empty word.

We will make use of the following infinitary variant of the Graham--Rothschild
Theorem~\cite{Graham1971} which is a direct consequence of the Carlson--Simp\-son
theorem \cite{carlson1984}.  This theorem was also obtained by Voigt around
1983 in a manuscript which, to our knowledge, was never published (see,
i.e.,~\cite[Theorem A]{promel1985baire}, \cite{carlson1987infinitary}).
\begin{theorem}
\label{thm:multCS}
Let $\Alphabet$ be a finite alphabet and $k\geq 0$ a finite integer.
If the set $\VSpace{\Alphabet}{\omega}{k}$ is coloured by finitely many colours, then there exists
an infinite-parameter word $W$ such that $W\left(\VSpace{\Alphabet}{\omega}{k}\right)$ is monochromatic.
\end{theorem}
We will also make use of the following finite version of Theorem~\ref{thm:multCS}.
\begin{theorem}
\label{thm:multCSfin}
Let $\Alphabet$ be a finite alphabet, $0\leq k\leq n$ and $r>0$ finite integers.
Then there exists $N=N(|\Alphabet|,k,n,r)$ such that
for every $r$-colouring of $\VSpace{\Alphabet}{N}{k}$ there exists a
word $W\in\VSpace{\Alphabet}{N}{n}$ such that $W\left(\VSpace{\Alphabet}{n}{k}\right)$ is monochromatic.
\end{theorem}

\section{Envelopes and embedding types}\label{sec:envelopes}
Essentially all big Ramsey degree results are based on a notion of envelope and embedding
type introduced by Laver and Milliken, see~\cite[Section 6.2]{todorcevic2010introduction}.
Precise definitions depend on the notion of a subspace (or a subtree). The following introduces
these concepts in the context of parameter spaces.
\begin{definition}
\label{def:envelope}
Given a finite alphabet $\Alphabet$, a set $S$ of parameter words in alphabet $\Alphabet$ and a parameter word $W$ in alphabet $\Alphabet$, we say that $W$ is an \emph{envelope} of $S$ if for every $U\in S$, there exists a parameter word $U'$ such that $W(U')=U$.
We call the envelope $W$ \emph{minimal} if there is no envelope of $S$ with fewer parameters than $W$.
\end{definition}
\begin{example}
\label{ex1}
Consider $\Alphabet=\{0\}$
The set $S=\{0,000\}\subseteq \VSpace{\Alphabet}{\omega}{0}$ has two minimal envelopes: $0\lambda_0\lambda_0$ and $0\lambda_00$.
Parameter word $\lambda_0\lambda_1\lambda_2\lambda_3$ is also an envelope of $S$, but it is not a minimal envelope.
\end{example}
\begin{prop}
\label{prop:subspace}
Let $\Alphabet$ be a finite alphabet, let $k\geq 0$ be a finite integer, let $S$ be a non-empty finite set of finite parameter words in alphabet $\Alphabet$ with each $U\in S$ having at most $k$ parameters and let $W$ be a minimal envelope of $S$. Then $W$ has at most $(|\Alphabet|+k)^{|S|}+|S|-|\Alphabet| - 1$ parameters.
Moreover, for every parameter $\lambda_i$ of $W$ and every minimal envelope $W'$ of $S$ it holds that the first occurrence of $\lambda_i$ has the same position in $W$ and $W'$.
\end{prop}
\begin{proof}
Fix $\Alphabet$, $k$, and $S$. Assume that $\Alphabet$ does not contain symbols $\dagger$ and $\ast$ which we will use with special meaning later.
Put $S=\{V^0,V^1,\ldots, V^{\ell-1}\}$.
We now show a method for constructing an envelope $W$.

	Put $m=\max_{0\leq i< \ell}(|V^i|)$ and for every $0\leq i< m$ define the \emph{slice $s^i$}, as the sequence (word) of length $\ell$ where we put
	$$s^i_j=\begin{cases}
		V^j_i & \hbox{if $i<|V^j|$,}\\
		\dagger & \hbox{if $i=|V^j|$,}\\
		\ast & \hbox{if $i>|V^j|$,}
	\end{cases}
	$$
	for every $j<\ell$.
Given $i$ and $j$ satisfying $0\leq i\leq j<m$ we say that slice $s^i$ is \emph{compatible} with slice~$s^j$ if for every $0\leq p<\ell$ it holds
that either $V^p_i=V^p_j$ or $V^p_j=\ast$. (In this case, the slice $s^j$ can be represented by the same parameter as the earlier slice $s^i$.)

Now construct a word $W$ of length $m$ by putting for every $0\leq j\leq m$
$$
W_j=
\begin{cases}
s & \hbox{if  slice $s^j$ consists only of $\ast$ and $s$ for some $s\in \Alphabet$,}\\
	W_{j'} & \hbox{if there exists $0\leq j'<j$ such that slice $s^{j'}$ is compatible}\\
&\hbox{with slice~$s^j$ and $j'$ is the minimal index with this property,}\\
\lambda_p & \hbox{otherwise, where $\lambda_p$ is the least so far unused parameter.}
\end{cases}
$$
It follows from the construction that $W$ is a parameter word and an envelope of $S$. We verify the minimality by checking that introduction of all parameters
is necessary.
Assume that $W_j$ is a first occurrence of parameter $\lambda_i$.  It follows that at least one of the following occurred:
	\begin{enumerate}
		\item There is a word $U\in S$ with $|U|=j$ (so slice $s^j$ contains $\dagger$).
		\item Slice $s_j$ either contains a parameter or at least two different letters in $S$, and there is no $j'<j$ such that slice $s^{j'}$ is compatible with slice $s^j$.
	\end{enumerate}
In both cases it holds that every envelope of $U$ must also introduce a new parameter $\lambda_i$, thereby giving the optimality of $W$ as well as the moreover part of the statement.

Notice that there are at most $|\Sigma+k|^{|S|}$ different slices not involving symbols $\ast$ and $\dagger$ and each slice not containing $\dagger$ is compatible with at least one of them. $|\Sigma|$ of these slices consist of a single letter $s\in \Sigma$ only and hence do not trigger the introduction of a new parameter.  There are at most $|S|-1$ slices containing $\dagger$ which always introduce a new parameter.
	 This leads to the upper bound of $(|\Alphabet|+k)^{|S|}+|S|-|\Alphabet| - 1$ parameters.
\end{proof}
\begin{definition}
\label{def:type}
Given a finite alphabet $\Alphabet$, a finite integer $k\geq 0$, a set $S$ of parameter words in alphabet $\Alphabet$ and an envelope $W$ of $S$, the \emph{embedding type} of $S$ in $W$, denoted by $\tau_W(S)$, is the set of parameter words such that $W(\tau_W(S))=S$.
\end{definition}
\begin{example}
The set $S=\{0,000\}$ has embedding type $\{\emptyset,0\}$ in both minimal envelopes given in Example~\ref{ex1}.
\end{example}

\begin{corollary}
\label{cor:types}
Let $\Alphabet$ be a finite alphabet and let $k,\ell>0$ be finite integers. Then
\begin{enumerate}
\item the set
$$\{\tau_W(S):S\subseteq \VSpace{\Alphabet}{\omega}{k}, |S|=\ell, W\hbox{ is a
minimal envelope of S}\}$$ is finite, and,
\item for every finite set  $S\subseteq \VSpace{\Alphabet}{\omega}{k}$ and its minimal envelopes
$W$ and $W'$ it holds that $\tau_W(S)=\tau_{W'}(S)$.
\end{enumerate}
\end{corollary}
\begin{proof}
	The first statement follows from the fact that there is an upper bound (given by Proposition~\ref{prop:subspace}) on the number of parameters of minimal envelopes and thus also on the length of words in the sets $\tau_W(S)$, $S\in \VSpace{\Alphabet}{\omega}{k}$.

	To see the second statement, consider $W$ and $W'$ to be minimal envelopes of a given set $S\in \VSpace{\Alphabet}{\omega}{k}$.
	Let $V$ be some word in $S$. Let $U\in \tau_W(S)$ be a word such that $W(U)=V$ (which exists since $W$ is an envelope). Clearly $|U|=i$ where $i$ satisfies $W_{|V|} = \lambda_i$ or $|U|$ is the number of parameters of $W$ if $|V|=|W|$.
	For every $i'<|U|$ we have $U_{i'}=V_{j'}$ where $j'$ is minimal such that $W_{j'}=\lambda_{i'}$. It follows that $U$ is unique.

	Now let $U'\in \tau_{W'}(S)$ be a word such that $W'(U')=V$. By the same argument as above we get that $U'$ is unique, and using the moreover part of Proposition~\ref{prop:subspace} it follows that $|U|=|U'|$ and $U_i=U'_i$ for every $i<|U|$.
\end{proof}
As a consequence of Corollary~\ref{cor:types}, we can use $\tau(S)$ for
$\tau_W(S)$ where $W$ is some (any) minimal envelope of $S$.
\begin{remark}
Our Definitions~\ref{def:envelope} and \ref{def:type} are closely related to
the definition of envelopes and types used by Dodos, Kanellopoulos and
Tyros~\cite{dodos2014} and by F\" urstenberg and
Katznelson~\cite{furstenberg1989}, see also~\cite[Chapter~5]{dodos2016}.  The
main difference is however the use of subspaces defined by variable words
rather than parameter words.  With respect to this notion of subspaces, the
dimension of minimal envelopes and thus also the number of embeddings types is not bounded by the
size of the set.
\end{remark}

\section{Big Ramsey degrees}\label{sec:bigdegrees}
In this section we prove Theorems~\ref{thm:posets} and \ref{thm:trianglefree}.
We start with Theorem~\ref{thm:trianglefree} and later show that Theorem~\ref{thm:posets} follows by very similar arguments.
\subsection{Triangle-free graphs}
\label{sec:trianglefree}
In this section we consider graphs to be structures in a language consisting of a single binary relation $E$.
We fix alphabet $\Alphabet=\{0\}$.

\begin{definition}
\label{def:trianglefree}
We define graph $\str{G}$ as follows:
\begin{enumerate}
  \item The vertex set $G$ is $\VSpace{\Alphabet}{\omega}{1}$ (that is, the set of all finite 1-parameter words).
  \item Given two vertices $U$ and $V$ such that $|U|<|V|$, we put an edge between $U$ and $V$ if and only if
\begin{enumerate}[label=(\roman*)]
\item\label{item:passing} $V_{|U|}=\lambda_0$ and 
\item\label{item:parallel} for no $0\leq j<|U|$ it holds that $U_j=V_j=\lambda_0$.
\end{enumerate}
\end{enumerate}
There are no other edges.
\end{definition}
\begin{remark}
Condition \ref{item:passing} in Definition~\ref{def:trianglefree} is the passing number representation of the Rado graph
used by Sauer~\cite{Sauer2006} (see also~\cite[Theorem~6.25]{todorcevic2010introduction}). Condition~\ref{item:parallel} is similar to Dobrinen's parallel 1's criterion~\cite[Definition 3.7]{dobrinen2017universal}.
The notion of subtree (or a subspace) used here is however different from~\cite{Sauer2006} and \cite{dobrinen2017universal}.
\end{remark}
\begin{lemma}
$\str{G}$ is triangle-free.
\end{lemma}
\begin{proof}
Assume to the contrary that $U$, $V$ and $W$ form a triangle.
Without loss of generality we can assume that $|U|<|V|<|W|$.
Because there is an edge between $U$ and $V$, we know that $V_{|U|}=\lambda_0$.
Because there is an edge between $U$ and $W$, we know that $W_{|U|}=\lambda_0$.
This contradicts the existence of an edge between $V$ and $W$.
\end{proof}
The following follows directly from the definition of substitution:
\begin{observation}
\label{obs:preserve2}
Let $W$ be an infinite-parameter word. Then for every $U,\allowbreak V\in G$ it holds that $U$ is adjacent to $V$ if and only
if $W(U)$ is adjacent to $W(V)$.
\end{observation}
Let $\str{H}$ with $H=\omega$ be (an enumeration of) the universal homogeneous triangle-free graph.  We define the mapping $\varphi\colon\omega\to G$ by putting
$\varphi(i)=U$ where $U$ is a 1-parameter word of length $2i+1$ defined
by putting for every $0\leq j\leq i$ 
$$U_{2j}=\begin{cases}\lambda_0&\hbox{ if and only if $j=i$,} \\0 &\hbox{ otherwise,}\end{cases}$$
and for every $0\leq j'<i$ $$U_{2j'+1}=\begin{cases}\lambda_0&\hbox{ if and only if $\{j',i\}$ is an edge of $\str{H}$,} \\0 &\hbox{ otherwise.}\end{cases} $$
Notice that since $U_{2i}=\lambda_0$ it holds that $U$ is indeed an 1-parameter word.
It is easy to check:
\begin{observation}
The function $\varphi$ is an embedding $\varphi\colon \str{H}\to \str{G}$ and thus $\str{G}$ is a universal triangle-free graph.
\end{observation}

Now we prove Theorem~\ref{thm:trianglefree} in the following form:
\begin{theorem}
\label{thm:trianglefree2}
For every finite $k\geq 1$ and every finite colouring of induced subgraphs of $\str{G}$ with $k$ vertices there exists
	$f\in \Emb(\str{G},\str{G})$ such that the colour of every $k$-vertex subgraph $\str A$ of $f(\str{G})$ depends only on
$\tau(A)=\tau(f^{-1}[A])$. 
\end{theorem}
Observe that by Corollary~\ref{cor:types}, we obtain the desired finite upper bound on number of colours.
The proof is again structured similarly to Milliken and Laver's results, see~\cite[Section 6.3]{todorcevic2010introduction}:
by a repeated application of Theorem~\ref{thm:multCS}, we obtain the desired copy.
\begin{proof}
Fix $k$, a finite set of colours $S$, and an $S$-colouring $\chi$ of
subsets of $G$ of size $k$.  Let $T^0,T^1,\allowbreak\ldots,\allowbreak T^{N-1}$
be all possible embedding types of subsets of $G$ of size $k$ in their minimal envelopes (given by
Corollary~\ref{cor:types}).
For every $0\leq i\leq {N-1}$, put $n_i=\max\{|U|\colon U\in T^i\}$.

Choose an infinite-parameter word $W^0\in \Space{\Alphabet}{\omega}{\omega}$ arbitrarily.
We construct a sequence of infinite-parameter words $W^1,W^2,\ldots,W^N$ such that for every $0< i\leq N$ the following is satisfied:
\begin{enumerate}
 \item $W^i=W^{i-1}(Z^i)$ for some infinite-parameter word $Z^i$,
 \item There exists colour $c^i$ such that $$\chi(W^i(U(T^{i-1})))=c^i$$ for every $U\in \VSpace{\Alphabet}{\omega}{n_{i-1}}$.
\end{enumerate}
	Let $f\colon G\to G$ be an (injective) function defined by putting $f(U)=W^N(U)$ for every $U\in G$.
By Observation~\ref{obs:preserve2}, $f$ preserves both edges and non-edges and consequently is an embedding.

	Now fix graph $\str{A}$ with $k$ vertices and its copy $\widetilde{\str{A}}$ in $f(\str{G})$. Then there exists $i$ such that $\widetilde{A}$ has embedding type $T^i$. By the construction of $f$ we know that the colour of $\widetilde{\str{A}}$ is $c^i$.
	Consequently, colours of subgraphs of $f(\str{G})$ with $k$ vertices depend only on their embedding types as desired.

It remains to show the construction of $W^i$.  Assume that $W^{i-1}$ is constructed.
Let $$\chi^i\colon \VSpace{\Alphabet}{\omega}{n_{i-1}}\to S$$ be a colouring
given by $$\chi^i(U)=\chi(W^{i-1}(U(T^{i-1}))).$$  By Theorem~\ref{thm:multCS} there exists an infinite-parameter word $Z^i$ and colour $c^i$
satisfying that $\chi^i(Z^i(U))=c^i$ for every $U\in \VSpace{\Alphabet}{\omega}{n_{i-1}}$. Put $W^i=W^{i-1}(Z^i)$.
\end{proof}

\subsection{Partial orders}
\label{sec:posets}
Throughout this section, we fix a language $L$ with a single binary relation $\leq$. A partial order $(A,\leq_\str{A})$ is then an $L$-structure $\str{A}$ with vertex set $A$ and a binary relation $\leq_\str{A}$.
We also fix the alphabet $\Alphabet=\{L,X,R\}$. We will use the lexicographic order of words that is based on an ``not-so-alphabetic'' order $<_{\mathrm{lxr}}$  of $\Alphabet$ defined by: $L<_{\mathrm{lxr}} X<_{\mathrm{lxr}}R$. We define the following binary relation $\preceq$ on $\Alphabet^*$:
\begin{definition}
\label{def:poset}
For $w,w'\in \Alphabet^*$ we put $w\prec w'$ if and only if there exists $0\leq i<\min(|w|,|w'|)$ such that
\begin{enumerate}[label=(\roman*)]
\item\label{cnd:1} $(w_i,w'_i)=(L,R)$ and
\item\label{cnd:2} for every $0\leq j< i$ it holds that $w_j\leq_{\mathrm{lxr}}w'_j$.
\end{enumerate}
For $w\prec w'$ we denote by $i(w,w')$ the minimal $i$ satisfying the condition \ref{cnd:1} above.
We put $w\preceq w'$ if and only if either $w=w'$ or $w\prec w'$.

We denote by $\str{O}$ the structure with vertex set $O=\Sigma^*$ ordered by $\preceq$.  (Thus we put $u \leq_\str{O} v$ if and only if $u \preceq v.$)
\end{definition}
\begin{remark}

The intuitive meaning of the definition above is that for every $w\in \Alphabet^*$ and every $j<|w|$ the letter $w_j$
describes a position of the vertex $w$ compared to vertices $u\in \Alphabet^*$ satisfying $|u|=j$.  If $w<u$ then it appears ``to the left'' and
is denoted by $L$.  If $u<w$ then it appears ``to the right'' and is denoted by $R$, and if $u$ and $w$
are incomparable we use $X$.  
\end{remark}
\begin{prop}
\label{prop:pos}
The structure $\str{O}$ is a partial order.
\end{prop}
\begin{proof}
It is easy to see that $\preceq$ is reflexive and anti-symmetric.  We verify transitivity. Let $w\prec w'\prec w''$
 and put $i=\min(i(w,w'),i(w',w''))$. 

First assume that $i=i(w,w')$.  Then we have $w_i=L, w'_i=R$ which implies that $w''_i=R$.
For every $0\leq j<i$ it holds that $w_j\leq_{\mathrm{lxr}}w'_j\leq_{\mathrm{lxr}}w''_j$.
It follows that $w\preceq w''$ and $i(w,w'')\leq i$.

Now assume that $i=i(w',w'')$. Then we have $w'_i=L$, $w''_i=R$ and because $w'_i=L$ then also $w_i=L$.
Again for every $0\leq j<i$ it holds that $w_j\leq_{\mathrm{lxr}}w'_j\leq_{\mathrm{lxr}}w''_j$.
It also follows that $w\preceq w''$ and $i(w,w'')\leq i$.
\end{proof}
The key to our construction is the following:
\begin{lemma}
\label{lem:preserve1}
Let $W$ be an infinite-parameter word. Then for every $w,w'\in \Alphabet^*$ it holds that $w\preceq w'$ if and only
if $W(w)\preceq W(w')$.
\end{lemma}
\begin{proof}
This can be easily checked using the fact that for every $i>0$, $\lambda_i$ first occurs in $W$ after the first occurrence of $\lambda_{i-1}$.
\end{proof}

Recall that by $\str{P}=(P,\leq_\str{P})$ we denote the universal homogeneous partial order.  Without loss of generality, we can assume that $P=\omega$
and thus fix an (arbitrary) enumeration of $\str{P}$.
We define function $\varphi\colon \omega\to \Alphabet^*$ by mapping $j\in P$ to a word $w$ of length $2j+2$ defined as:
$$
(w_{2i},w_{2i+1})=
\begin{cases}
(L,L) & \hbox {whenever $i<j$ and $j\leq_\str{P} i$,}\\
(R,R) & \hbox {whenever $i<j$ and $i\leq_\str{P} j$,}\\
(X,X) & \hbox {whenever $i<j$ and $i$ is $\leq_\str{P}$-incomparable with $j$,}\\
(L,R) & \hbox {whenever $i=j$.}
\end{cases}
$$
\begin{prop} \label{prop:universalposet}
The function $\varphi$ is an embedding $\varphi\colon \str{P}\to \str{O}$.
Consequently, $\str{O}$ is a universal partial order.
\end{prop}
\begin{proof}
Given $i<j\in \omega$, put $u=\varphi(i)$ and $v=\varphi(j)$ and consider three cases:
\begin{enumerate}
\item $i\leq_\str{P} j\implies u\preceq v$:
We have $u_{2i}=L$ and $v_{2i}=R$ and we check that for every $0\leq k< i$ it holds that $u_{2k}\leq_{\mathrm{lxr}} v_{2k}$ and thus also $u_{2k+1}\leq_{\mathrm{lxr}} v_{2k+1}$.
If $u_{2k}=L$ then this follows trivially. If $u_{2k}=X$ then we know that $k$ is incomparable with $i$ by $\leq_\str{P}$. It follows that $v_{2k}\neq L$ because $i\leq_\str{P} j$ and thus it can not hold that $j\leq_\str{P} k$.
If $u_{2k}=R$ then we get $k\leq_\str{P} i\leq_\str{P} j$ and thus also $v_{2k}=R$.
\item $j\leq_\str{P} i\implies v\preceq i$:
Here we have $u_{2i+1}=R$ and $v_{2i+1}=L$.
Analogously as in the previous case, we can check that for every $0\leq k< i$ it holds that $v_{2k}\leq_{\mathrm{lxr}} u_{2k}$.
\item If $i$ is incomparable with $j$ in $\leq_\str{P}$ then $u$ is incomparable with $v$ in $\preceq$:
Assume the contrary and let $k\leq i$ be such that either $u_{2k}=L$ and $v_{2k}=R$ or $u_{2k+1}=L$ and $v_{2k+1}=R$.  Clearly $k<i$ because $v_{2i}=v_{2i+1}=X$.
We get that $i\leq_\str{P} k\leq_\str{P} j$ or $j\leq_\str{P} k\leq_\str{P} i$. A contradiction.
\end{enumerate}
\end{proof}
\begin{remark}
Easy constructions of universal partial orders are interesting in their own right, see~\cite{hedrlin1969universal,pultr1980combinatorial,Hubicka2005,Hubicka2005a,hubicka2011some}.
	Observe also that the lexicographic order $\leq_{\mathrm{lxr}}$ is a linear extension of $\preceq$ and thus the construction can be seen as a direct refinement
of the Laver--Devlin construction.
\end{remark}
Now we are ready to prove Theorem~\ref{thm:posets} in the following form.
\begin{theorem}
\label{thm:posets2}
For every finite $k\geq 1$ and every finite colouring of (induced) suborders of $\str{O}$ with $k$ elements, there exists
	$f\in \Emb(\str{O},\str{O})$ such that the colour of every suborder $\str{A}$ of $f(\str{O})$ with $k$ vertices depends only on
$\tau(A)=\tau(f^{-1}[A])$. 
\end{theorem}
\begin{proof}
This follows in an analogy to Theorem~\ref{thm:trianglefree2}.

Fix $k$ and a finite colouring $\chi$ of
subsets of $O$ of size $k$.  Let $T^0,T^1,\allowbreak\ldots,\allowbreak T^{N-1}$
be all possible embedding types of subsets of $O$ of size $k$ in their minimal envelopes (given by
Corollary~\ref{cor:types}).
For every $0\leq i\leq {N-1}$ put $n_i=\max\{|U|\colon U\in T^i\}$.

Choose infinite-parameter word $W^0\in \Space{\Alphabet}{\omega}{\omega}$ arbitrarily.
We construct a sequence of infinite-parameter words $W^1,W^2,\ldots,W^N$ such that for every $0< i\leq N$ the following is satisfied:
\begin{enumerate}
 \item $W^i=W^{i-1}(Z^i)$ for some infinite-parameter word $Z^i$,
 \item There exists colour $c^i$ such that $$\chi(W^i(U(T^{i-1})))=c^i$$ for every $U\in \VSpace{\Alphabet}{\omega}{n_{i-1}}$.
\end{enumerate}
Let $f$ be defined by $f(U)=W^N(U)$.
By Lemma~\ref{lem:preserve1} we know that this is an embedding
with the desired properties.

Word $W^i$ is again constructed by an application of Theorem~\ref{thm:multCS}.
\end{proof}
\section{Ramsey classes (of finite structures)}\label{sec:smallclasses}
While determining exact big Ramsey degrees is technically challenging, we can use
the techniques from this paper to determine exact small Ramsey degrees. In this section, we discuss in detail
how exact small Ramsey degrees can be obtained for both triangle-free graphs (where we obtain a new proof of this old result of Nešetřil and Rödl~\cite{Nevsetvril1977,Nevsetvril1989}) and
partial orders (where we systematically arrive to the proof strategy identified by Ma{\v{s}}ulovi{\'c} in 2018~\cite{masulovic2016pre}). This provides an explicit link between the techniques for giving bounds on small and big Ramsey degrees,
problems previously studied independently, and also a method that can be applied to other structures as well (for example to metric spaces~\cite{balko2021big}).

The exact big Ramsey degrees require more effort, see Section~\ref{sec:concluding}, and are determined in follow-up papers~\cite{Balko2021exact,Balko2023}. Curiously, these papers show that triangle-free graphs differ from partial orders in a subtle yet important way. While big Ramsey degrees of triangle-free graphs can be obtained by analysing the proof given here, for a precise characterisation of big Ramsey degrees of partial orders, a refinement of parameter spaces needed to be introduced.  The issue is demonstrated in Section~\ref{finitepartial} as the asymmetry between the use of letters $L$ and $R$ in the construction.

\subsection{Ordered triangle-free graphs}
An \emph{ordered graph} is a relational structure $\str{A}$ in a language consisting of two binary relations $E$ and $\leq$
such that $(A,E_\str{A})$ is a graph and $(A,\leq_\str{A})$ is a linear order.

We prove a special case of the Ne\v set\v ril--R\" odl theorem~\cite{Nevsetvril1977,Nevsetvril1989}.
Our proof is based on the ideas developed in the previous sections and is arguably the most direct proof
of this result known to date, giving a particularly simple description of the Ramsey graph $\str{C}$.
We shall remark that similar constructions have been known for unrestricted classes, see~\cite[Theorem 12.13]{PromelBook}
for a proof of the Ramsey property of the class of all finite ordered graphs.
However, to our best knowledge, a similar strategy has been applied to a class of graphs with a forbidden subgraph
in special cases only (for colouring vertices and edges~\cite{nevsetril1975ramsey,nesetril1975type}).
\begin{theorem}[Special case of Ne\v set\v ril--R\"odl~\cite{Nevsetvril1977,Nevsetvril1989}]
\label{thm:fingraphs}
For every integer $r>0$ and every pair of finite ordered triangle-free graphs $\str{A}$ and $\str{B}$, there exists a finite ordered
triangle-free graph $\str{C}$ such that $\str{C}\longrightarrow (\str{B})^\str{A}_{r,1}$.
\end{theorem}
\begin{proof}
We fix alphabet $\Alphabet=\{0\}$.
Recall the graph $\str{G}$ defined in Definition~\ref{def:trianglefree}.
By $\str{G}_N$, we denote the ordered graph created from $\str{G}$ by considering only vertices in $\VSpace{\Alphabet}{N}{1}$ 
and adding a lexicographic ordering of the vertices (where we consider vertices to be strings in alphabet $\{0,\lambda_0\}$ ordered $0<\lambda_0$).

We will show that for sufficiently large $N$ (to be specified at the end of the proof) it holds that
$\str{G}_N\longrightarrow (\str{B})^{\str{A}}_{r,1}$.
Towards this, we first define a more careful way to embed an ordered triangle-free graph $\str{B}$ into $\str{G}_n$ for some $n$ to be determined.
\medskip

Let $\str{B}$ be an ordered triangle-free graph.
 For simplicity we can assume
that $B=\{0,1,\ldots,|B|-1\}$ and that $\leq_\str{B}$ coincides with the order of integers.  
We define an embedding $\varphi\colon \str{B}\to\str{G}_n$ for some
sufficiently large $n$ to be fixed later by the following procedure.  We say
that a function $f\colon B\to\{0,\lambda_0\}$ is a \emph{1-type} (extending $\str{B}$) if $\str{B}$ extended
by a new vertex which is adjacent precisely to those vertices $v\in B$ satisfying
$f(v)=\lambda_0$ is a triangle-free graph. In other words, there are no two adjacent vertices $v,v'\in B$ such that $f(v)=f(v')=\lambda_0$.

Now enumerate all possible 1-types as $f_0,\allowbreak f_1,\allowbreak\ldots,\allowbreak f_{d-1}$ ordered
lexicographically with respect to $\leq_\str{B}$.   More precisely, we see every function $f_i$ as a word $w^i$ of length $|B|$ with $w^i_j=f_i(j)$
and order those words lexicographically.

Put $\varphi(v)=V$ where word $V$ is defined as follows: $|V|=d+v$ and 
$$
V_j=
\begin{cases}
f_j(v) & \hbox{for $j< d$},\\
\lambda_0 & \hbox {for $d\leq j<d+v$ such that $v$ is adjacent to $j-d$ in $\str{B}$},\\
0 & \hbox {for $d\leq j<d+v$ such that $v$ is not adjacent to $j-d$ in $\str{B}$}.
\end{cases}
$$

Now put $n=d+|B|$. It is easy to see that $\varphi$ is an embedding of $\str{B}$ to $\str{G}_n$ (to see that the order is preserved, note that all extensions by a vertex connected to precisely one vertex of $\str B$ are triangle-free).
\begin{figure}
\centering
\includegraphics{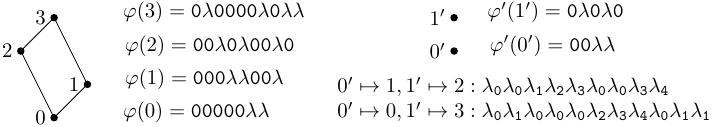}
\caption{Representation of a graph $\str{B}$ ordered naturally $0\leq_\str{B}1\leq_\str{B}2\leq_\str{B}3$
and a graph $\str{A}$ ordered $0'\leq_\str{A} 1'$ along with a parameter word representing all embeddings of $\str{A}$ to $\str{B}$ as constructed in the proof of Claim~\ref{clm2}. For easier reading, $\lambda_0$ is typeset as $\mathrm{\lambda}$.}
\label{fig:trianglefree}
\end{figure}
An example of this representation is depicted in Figure~\ref{fig:trianglefree}.

\medskip

Let $\varphi'$ be an embedding of $\str{A} \to \str{G}_k$ for some $k>0$ constructed in the same way as above. 
With $n$ and $k$ fixed we show:

\begin{claim}
\label{clm2}
	Given $\theta \in \Emb(\str{A},\str{B})$ denote by $\widetilde{\str{A}}$ the copy $\theta(\str{A})$.
	Then there exists a $k$-parameter word $W\in \VSpace{\Alphabet}{n}{k}$ such that $W(\varphi'(A))=\varphi(\widetilde{A})$.
\end{claim}
Let $f_0,f_1,\ldots, f_{d-1}$ be the enumeration of 1-types of $\str{B}$ in the lexicographic
order and let $f'_0,f'_1,\ldots, f'_{d'-1}$ be the enumeration of 1-types of $\widetilde{\str{A}}$ also ordered lexicographically.
Let $h\colon \{0,1,\ldots,d-1\}\to\{0,1,\ldots,d'-1\}$ be the mapping such that for every $i\in \{0,1,\ldots,d-1\}$ function $f_i$ restricted to $\widetilde{A}$ is $f'_{h(i)}$.
Observe that every 1-type $f$ of
$\widetilde{\str{A}}$ can be extended to a 1-type $f'$ of
$\str{B}$ by putting $f'=f(v)$ for $v\in \widetilde{A}$ and $f'(v)=0$
otherwise.
It follows that $h$ exists and is surjective. 

For every $v\in B\setminus \widetilde{A}$ we put $e(v)$ to be an integer
such that $f'_{e(v)}$ describes the neighbourhood of $v$ in $\widetilde{A}$.

We now define a string $W$ (which we later verify to be a parameter word) of length $d+\max({\widetilde{A}})$ as follows:
$$
W_j=
\begin{cases}
\lambda_{h(j)} & \hbox{for every $0\leq j< d$,}\\
\lambda_{d'+\theta^{-1}(j-d)} & \hbox{for every $d\leq j$ such that $j-d\in \widetilde{A}$,}\\
\lambda_{e(j-d)} & \hbox{for every $d\leq j$ such that $j-d\notin \widetilde{A}$}.
\end{cases}
$$
First observe that $W_0=\lambda_0$.  This is because $f_0$ and $f'_0$ are both constant zero functions.

We verify that $W$ is a $k$-parameter word, that is,
for every $1\leq j<k$ it holds that the first occurrence of $\lambda_j$ comes after the first occurrence of $\lambda_{j-1}$. We consider three cases:
\begin{enumerate}
\item
$j< d'$: Function $f'_j$ can be extended to function $f''_j\colon B\to \{0,\lambda_0\}$ by putting $f''_j(v)=0$ for every $v\notin \widetilde{A}$.
This is clearly a 1-type of $\str{B}$ and therefore there exists $j'$ such that $f''_j=f_{j'}$.
From this it follows that $W_{j'}=\lambda_j$. Because zero is the minimal element of the alphabet we get that this is also the first occurrence of $\lambda_j$ in $W$. Finally,
because the first occurrence of $\lambda_{j-1}$ can be found same way and 
the extension by zeros preserves the relative lexicographic order, we know that $\lambda_j$ appears after $\lambda_{j-1}$.
\item $j=d'$:
$\lambda_j$ occurs once at position $d+\theta(j-d')=d+\theta(0)$. We already checked that $\lambda_{j-1}$ occurs before $d$.
\item $d'< j<k$:
For every $d'< j< k$ it holds that $\lambda_j$ occurs precisely once at position $d+\theta(j-d')$
so the desired ordering follows from the monotonicity of $\theta$.
\end{enumerate}
This finishes the proof that $W$ is indeed $k$-parameter word.
By substituting $\varphi'(A)$ into $W$ it can be also checked that $W(\varphi'(A))=\varphi(\widetilde{A})$. This finishes the proof of Claim~\ref{clm2}.

\medskip

Now let $N=N(1,k,n,r)$ be given by Theorem~\ref{thm:multCSfin}. We claim that $\str{G}_N\longrightarrow (\str{B})^{\str{A}}_{r,1}$.
Consider an $r$-colouring of $\Emb(\str{A},\str{G}_N)$. Observe that for every $W\in \VSpace{\Alphabet}{N}{k}$ we get a unique copy of $\str{A}$
in $\str{G}_N$ given by $W(\varphi'(A))$. We thus obtain an $r$-colouring of $\VSpace{\Alphabet}{N}{k}$ and by an application of
Theorem~\ref{thm:multCSfin} a word $\widetilde{W}\in \VSpace{\Alphabet}{N}{n}$ for which this colouring is constant.
The monochromatic copy of $\str{B}$ is now given by $\widetilde{W}(\varphi(B))$: By Claim~\ref{clm2} we know that every copy of $\str A$ in $\str B$ is induced by a word from $\VSpace{\Alphabet}{N}{k}$, and so all of them indeed have the same colour.
\end{proof}
\subsection{Partial orders with linear extension}
\label{finitepartial}
Now we will consider structures in language with two binary relations $\leq$ and $\trianglelefteq$.
$\str{A}$ is a \emph{partial order with linear extension} if $(A,\trianglelefteq_\str{A})$ a partial order
and $(A,\leq_\str{A})$ is its linear extension.

We prove:
\begin{theorem}[\cite{Nevsetvril1984,Trotter1985}]
\label{thm:finposets}
For every integer $r>0$ and every pair of finite partial orders with linear extensions $\str{A}$ and $\str{B}$ there exists a finite partial order with linear extension
$\str{C}$ such that $\str{C}\longrightarrow (\str{B})^\str{A}_{r,1}$.
\end{theorem}
\begin{remark}
The proof of Theorem~\ref{thm:finposets} presented here is related to proofs of this result based on the Graham--Rothschild theorem \cite[Theorem~4.1]{masulovic2016pre}.  We present it because our representation of the partial order by finite words is different. This difference is necessary to show Theorem~\ref{thm:posets} (where countably infinite partial orders need to be represented), but also perhaps makes the proof of Theorem~\ref{thm:finposets} a bit more systematic.
\end{remark}
\begin{proof}
We fix alphabet $\Alphabet=\{L,X,R\}$ and its ordering $L<_{\mathrm{lxr}} X<_{\mathrm{lxr}}R$.
Denote by $\str{O}_N$ the partial order induced on $\VSpace{\Alphabet}{N}{0}$ by $\str{O}$ (given by Definition~\ref{def:poset})
 with a linear extension defined by the lexicographic order.

Fix $\str{A}$ and $\str{B}$ and proceed in analogy to the proof of Theorem~\ref{thm:fingraphs}.  
 For simplicity, we can assume
that $B=\{0,1,\ldots,|B|-1\}$ and that $\leq_\str{B}$ coincides with the order of integers.  
We show that there exists $N$ such that
$\str{O}_N\longrightarrow (\str{B})^\str{A}_{r,1}$.

\medskip

We define an embedding $\varphi\colon \str{B}\to\str{O}_n$ for some
sufficiently large $n$ (to be fixed later) by the following procedure.  
We say
that function $f\colon B\to\{L,X\}$ \emph{represents a downset} of $\str{B}$ if the set $\{v\colon f(v)=L\}$ is downwards
closed with respect to $\trianglelefteq_\str{B}$.

Now enumerate all possible functions representing a downset as $f_0,\allowbreak f_1,\allowbreak \ldots,\allowbreak  f_{d-1}$ ordered
lexicographically with respect to $\leq_\str{B}$ and $<_{\mathrm{lxr}}$.  Put
$\varphi(v)=w$ where $w$ is a word of length $d+v+1$ defined as follows:
$$w_j=\begin{cases}f_j(v)&\hbox{ for $0\leq j<d$,}\\
R&\hbox{ for $d\leq j<d+v$,}\\
L&\hbox{ for $j=d+v$.}
\end{cases}$$
\begin{figure}
\centering
\includegraphics{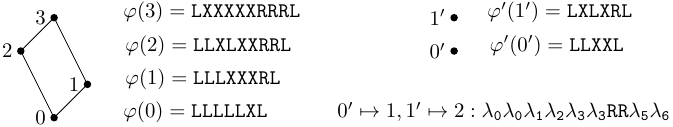}
\caption{Representation of a partial order $\str{B}$ with natural linear extension $0\leq_\str{B}1\leq_\str{B}2\leq_\str{B}3$
and a partial order $\str{A}$ with linear extension $0'\leq_\str{A} 1'$ (relations $\trianglelefteq_\str{B}$ and $\trianglelefteq_\str{A}$ are depicted
by Hasse diagrams) along with a parameter word representing the embedding of $\str{A}$ to $\str{B}$ as constructed in the proof of Claim~\ref{clm3}.}
\label{fig:posets}
\end{figure}
An example of this representation is depicted in Figure~\ref{fig:posets}.

Now put $n=d+|B|+1$. It is easy to see that $\varphi$ is an embedding of $\str{B}$ to $\str{O}_n$:
levels $d$ to $d+|B|$ code the linear extensions given by $\leq_\str{B}$
while earlier levels code all downsets.  Every pair of vertices $u\leq_\str{B} v$ which are not
comparable by $\trianglelefteq$ have downsets witnessing this which makes sure that their images are also not comparable
by $\preceq$.

\medskip

Let $\varphi'$ be an embedding of $\str{A} \to \str{O}_k$ for some $k>0$ constructed in the same way as $\varphi$. 

\begin{claim}
\label{clm3}
For every $\theta \in \Emb(\str{A},\str{B})$ there exists a $k$-parameter word
$W\in \VSpace{\Alphabet}{n}{k}$ such that $W(\varphi'(A))=\varphi(\theta[A])$.
\end{claim}
Put $\widetilde{\str{A}}=\theta(\str{A})$.
Let $f_0,f_1,\ldots, f_{d-1}$ be the enumeration of functions representing downsets of $\str{B}$ in the lexicographic
order (with respect to $\leq_\str{B}$ and $<_{\mathrm{lxr}}$) and $f'_0,f'_1,\ldots, f'_{d'-1}$ be the enumeration of functions representing downsets of $\widetilde{\str{A}}$ also ordered lexicographically.
Let $h\colon \{0,1,\ldots,d-1\}\to\{0,1,\ldots,d'-1\}$ be the mapping such that $f_i$ restricted to $\widetilde{A}$ is $f'_{h(i)}$.
Observe that every downset $f$ of
$\widetilde{\str{A}}$ can be extended to a downset of $f'$ and thus $h$ is well defined and surjective.

	We now define a string word $W$ (which we later verify to be a parameter word) of length $d+\max({\widetilde{A}})$ as follows:
$$
W_j=
\begin{cases}
 \lambda_{h(j)} & \hbox{for every $0\leq j< d$,}\\
 \lambda_{d'+\theta^{-1}(j-d)} & \hbox {for every $d\leq j< |B|+1$ such that $j-d\in \widetilde{A}$},\\
 R & \hbox{for every $d\leq j< |B|+1$ such that $j-d\notin \widetilde{A}$.}
\end{cases}
$$
Next, we verify that $W$ is a $k$-parameter word.  For this, we need to find
for every $f'_j$ its lexicographically minimal extension $f_{j'}$ and verify that the lexicographic order is preserved.
Given $f'_j$, we construct function $f\colon \str{B}\to\{L,X,R\}$ by putting:
$$
f(v)=
\begin{cases}
f'_j(v)& \hbox{if $v\in\widetilde{A}$,}\\
X & \hbox{if $v\notin\widetilde{A}$ and there exists $u\in\widetilde{A}$, $f'_j(u)=X$ and $u\trianglelefteq_\str{B} v$,}\\
L & \hbox{otherwise.}
\end{cases}
$$

Observe that there is $j'$ such that $f=f_{j'}$ and that $f_{j'}$ is lexicographically minimal among all functions $f_\ell$ which represent a downset of $\str B$ such that $h(f_\ell)=f'_j$. This is due to the fact that 
we put $f(v)=X$ only when this was forced by a ``witness'' $u\in\widetilde{A}$ for which $f'_j(u)=X$, and
thus the value of $v$ is $X$ in every extension of $f'_j$ which represents a downset.
To see that this construction preserves the lexicographic order, pick arbitrary $1\leq i < j \leq d'-1$ and let $f^i$ and $f^j$ be the extensions of $f'_i$ respectively $f'_j$ constructed as above. Let $v\in \widetilde{A}$ be minimal (with respect to $\leq_\str{B}$) such that $f'_i(v) \neq f'_j(v)$. We have that $f'_i(v) = L$ and $f'_j(v) = X$. Let $v'\in \widetilde{A}$ be minimal such that $f^i(v')\neq f^j(v')$. If $v'\in\widetilde{A}$ then $v=v'$ and indeed $f^i$ is lexicographically smaller than $f^j$. So $v'\notin\widetilde{A}$ and we know that exactly one of $f^i(v')$ and $f^j(v')$ is equal to $L$. If $f^i(v') = L$ then, again, $f^i$ is lexicographically smaller than $f^j$. So $f^i(v') = X$ and $f^j(v') = L$. This means that there is $u\in\widetilde{A}$ with $u\trianglelefteq_\str{B} v'$ such that $f'_i(u)=X$ and $f'_j(u)=L$. However, $u\trianglelefteq_\str{B} v'$, implies $u\leq_\str{B} v'$, hence $u \leq_\str{B} v$, which contradicts minimality of $v$. Hence this construction indeed preserves the lexicographic order. Note that at this moment we make use of the fact that our representation uses downsets rather than all 1-type
which would seem to be a more direct analogy of the proof of Theorem~\ref{thm:fingraphs}.

Therefore $W$ is indeed a parameter word, which finishes the proof of Claim~\ref{clm3}.

\medskip

Now let $N=N(0,k,n,r)$ be given by Theorem~\ref{thm:multCSfin}. We claim that $\str{O}_N\longrightarrow (\str{B})^{\str{A}}_{r,1}$.
Consider an $r$-colouring of $\str{O}_N$. Observe that for every $W\in \VSpace{\Alphabet}{N}{k}$ we get a unique copy of $\str{A}$
in $\str{O}_N$ given by $W(\varphi'(A))$. We thus obtain an $r$-colouring of $\VSpace{\Alphabet}{N}{k}$ and by application of
Theorem~\ref{thm:multCSfin} a word $\widetilde{W}\in \VSpace{\Alphabet}{N}{n}$ for which this colouring is constant.
The monochromatic copy of $\str{B}$ is now given by $\widetilde{W}(\varphi(B))$: By Claim~\ref{clm2} we know that every copy of $\str A$ in $\str B$ is induced by a word from $\VSpace{\Alphabet}{N}{k}$, and so all of them indeed have the same colour.
\end{proof}
\section{Applications}\label{sec:applications}
In this section, we briefly discuss some examples of structures where
finiteness of big Ramsey degrees follows as a direct consequence of
Theorems~\ref{thm:posets} and~\ref{thm:posets2}. This includes some already
known examples (linear orders, graphs, triangle-free graphs, ultrametric
spaces) as well as a new example ($S$-metric spaces). 

For each of the examples we will construct an interpretation in the universal homogeneous partial order $\str{P}$ (or
its fixed linear extension) which has the property that vertices of this
interpretation are formed by $\Emb(\str{V},\str{P})$ for some finite poset
$\str{V}$.  By obtaining a common representation of these structures within
partial orders we also show that free superpositions of such structures have
finite big Ramsey degrees, thereby giving a partial answer to a question asked
by Zucker during the 2018 BIRS workshop ``Unifying Themes in Ramsey Theory.''

We stress that the representations here generally only lead
to very generous upper bounds on big Ramsey degrees.

\subsection{Triangle-free graphs}
It may be a bit of a surprise that Theorem~\ref{thm:posets} implies
Theorem~\ref{thm:trianglefree} in a particularly easy way.  Given the universal homogeneous
partial order $\str{P}$, we denote by $\str{G}_\str{P}$ the following graph:
\begin{enumerate}
 \item Vertices of $\str{G}_\str{P}$ are all triples of distinct vertices $(u_0,u_1,u_2)$ of $\str{P}$ such that $u_0<_\str{P} u_2$, while $(u_0,u_1)$ and $(u_1,u_2)$ are incomparable in $\str{P}$.
 \item Vertices $(u_0,u_1,u_2)$ and $(v_0,v_1,v_2)$ form an edge of $\str G_\str P$ if and only if $u_0<_\str{P} v_1<_\str{P} u_2$, $v_0<_\str{P} u_1<_\str{P} v_2$
and all other pairs $(u_i,v_j)$, $i,j\in\{0,1,2\}$, are incomparable in $\str{P}$.
\end{enumerate}
By transitivity, $\str{G}_\str{P}$ is triangle-free: if both
$\{(u_0,u_1,u_2),(v_0,v_1,v_2)\}$ and $\{(v_0,v_1,v_2),\allowbreak
(w_0,w_1,w_2)\}$  are edges of $\str{G}_\str{P}$ then we have $u_0\leq_\str{P}
w_2$ which implies that $\{(u_0,u_1,\allowbreak u_2),\allowbreak (w_0,w_1,w_2)\}$ is a
non-edge.

It is not hard to check that there is an embedding $\varphi$ from the homogeneous universal triangle-free graph $\str{H}$ to $\str{G}_\str{P}$.
Recall that the vertex set of $\str{H}$ is $\omega$ and construct the embedding $\varphi$ inductively.
For each vertex $i\in \omega$ assume that $\varphi(i')$ is constructed for every $i'<i$ and apply the extension property of $\str{P}$ to obtain
three disjoint vertices $i_0,i_1,i_2\in P$, such $(i_0,i_1,i_2)$ is a vertex of $\str{G}_\str{P}$, and for every $j\leq i$ vertices $\varphi(j)=(j_0,j_1,j_2)$ are disjoint from $(i_0,i_1,i_2)$ and the following is satisfied:
\begin{enumerate}
\item If $i,j$ forms an edge of $\str{H}$ put $i_0\leq_\str{P} j_1\leq_\str{P} i_2$ and $j_0,\leq_\str{P} i_1,\leq_\str{P} j_2$ so $(i_0,i_1,i_2)$ and $(j_0,j_1,j_2)$ forms an edge of $\str{G}_\str{P}$.
\item If $i,j$ does not form an edge of $\str{H}$ put $i_0\leq_\str{P} j_2$ and $j_0\leq_\str{P} i_2$ while keeping all other pairs $(i_k,j_k')$, $k\in \{0,1,2\}$ incomparable in $\str{P}$.
\end{enumerate}

To finish the proof of Theorem~\ref{thm:trianglefree}, assume that we are given a finite colouring of $\Emb(\str{A},\str{H})$ for some finite triangle-free graph $\str A$. Since $\str H$ is universal, it contains a copy of $\str G_\str P$ and hence it induces a colouring of $\Emb (\str{A},\str G_\str P)$. This can be turned into a finite colouring of substructures of $\str P$ on at most $3\lvert A\rvert$ vertices and hence, by Theorem~\ref{thm:posets}, there is a copy of $\str P$ with at most a bounded number of colours. This corresponds to a copy of $\str G_\str P$ in $\str G_\str P$ with at most a bounded number of colours, and the rest follows since $\str H$ embeds into $\str G_\str P$.

\subsection{Urysohn $S$-metric spaces}
Let $S$ be a set of non-negative reals such that $0\in S$. A metric space $\str{M}=(M,d)$ is an \emph{$S$-metric space}
if for every $u,v\in M$ it holds that $d(u,v)\in S$. We call a countable $S$-metric
space $\str{M}$ a \emph{Urysohn $S$-metric space} if it is homogeneous (that is, every isometry of its finite subspaces extends to a bijective isometry from $\str{M}$ to $\str{M}$) 
and every countable $S$-metric space embeds to it.
(For a more general definition of the Urysohn space and Urysohn sphere, see e.g.~\cite{lopez2008oscillation})
In the following, we will see $S$-metric spaces as relational structures in a language with a binary relation $R_\ell$ for every $\ell\in S\setminus \{0\}$.

A finite set of non-negative reals $S=\{0=s_0<s_1<\cdots<s_n\}$ is \emph{tight} 
if $s_{i+j}\leq s_i+s_j$ for all $0\leq i\leq j\leq i+j\leq n$ (see~\cite{masulovic2016pre}).
It follows from a classification by Sauer~\cite{Sauer2013} that for every such $S$ there exists a Urysohn $S$-metric space.

Ma\v sulovi\'c in~\cite[Theorem 4.4]{masulovic2016pre} shows a way to represent $S$-metric spaces for every finite tight set $S$
by partial orders. Using this construction we obtain:
\begin{corollary}
\label{cor:urysohn}
Let $S$ be a finite tight set of non-negative reals. Then the Urysohn $S$-metric space has finite big Ramsey degrees.
\end{corollary}
We will show a special case of Corollary~\ref{cor:urysohn} where $S=\{0,1,\ldots, d\}$. For other tight sets, we refer the reader to \cite[Theorem 4.4]{masulovic2016pre}.
\begin{proof}[sketch]
Fix $d$ and $S=\{0,1,\ldots, d\}$. Construct an $S$-metric space $\str{M}_S$ as follows:
\begin{enumerate}
 \item Vertices are chains of vertices of $\str{P}$ of length $d$.
 \item Given two chains $u_1<_\str{P} \cdots <_\str{P} u_{d}$ and $v_1<_\str{P} \cdots <_\str{P} v_{d}$ their distance is the minimal $\ell\in \{0,1,\ldots d\}$ such that
for every $i\in \{1,\ldots d-\ell\}$ it holds that $u_i<_\str{P} v_{i+\ell}$ and $v_i<_\str{P} u_{i+\ell}$.
\end{enumerate}
Just as in the case of triangle-free graphs, triangle inequality follows from transitivity and one can embed the Urysohn $S$-metric space to $\str{M}_S$ using an on-line algorithm, hence Corollary~\ref{cor:urysohn} follows.
\end{proof}
Note that not all finite sets $S$ for which there exists a Urysohn $S$-metric space
(these were characterised by Sauer~\cite{Sauer2013}) are tight and thus
Corollary~\ref{cor:urysohn} is not a complete characterisation.

Just like Theorem~\ref{thm:posets}, Corollary~\ref{cor:urysohn} has a known finite form.  The Ramsey property of the
class of all finite ordered metric spaces was shown by Ne\v set\v
ril~\cite{Nevsetvril2007} (see also~\cite{Dellamonica2012} for graph metric spaces). This result was later
generalised to all $S$-metric spaces~\cite{Hubicka2016, Konecny2018b}.

\subsubsection{Oscillation stability of the Urysohn sphere}
A structure is called \emph{(weakly) indivisible} if its big Ramsey degree of a vertex is equal to 1\footnote{Some authors consider indivisible structures to have the property that whenever their vertex set is partitioned into two parts, one of them is isomorphic to the original structure~\cite{cameron1997random} and ``weakly'' signifies the form as discussed here~\cite{Sauer2014}.}.
Work on the indivisibility of homogeneous $S$-metric spaces was originally
motivated by a connection to oscillation stability of the \emph{Urysohn sphere $\str{S}$} (up to isomorphism the unique complete separable ultrahomogeneous
metric space with diameter 1 into which every separable metric space with diameter
less or equal to 1 embeds isometrically).

In our setting, this can be formulated as a question about \emph{approximate indivisibility}: for every
finite colouring of vertices of $\str{S}$ and every $\epsilon>0$ there exists a colour $c$
and a copy $\str{S}'$ of $\str{S}$ in $\str{S}$ such that for every point $p\in S'$ there
exists $p'$ of colour $c$ in distance at most $\epsilon$.
This can be seen as a Urysohn-sphere analog of the distortion problem for $\ell_2$
which itself originated in the 1970s in the work of Milman~\cite{milman1971new,milman1992dvoretzky}.

Lopez-Abad and Nguyen van Th\'e showed that approximate indivisibility of
$\str{S}$ can be reduced to the question about indivisibility of homogeneous
$S$-metric spaces for every $S$ being a finite initial segment of
integers~\cite{lopez2008oscillation}.  The indivisibility was later proved
by Nguyen van Th\'e and Sauer~\cite{NVT2009b}, thus resolving positively the question of
approximate indivisibility of $\str{S}$. 

The indivisibility results were developed further.  Indivisibility of
$S$-metric spaces was later shown by Sauer~\cite{Sauer2012}. See
also~\cite{Delhomme2007, delhomme2008indivisible,The2010} for more background
on vertex partition theorems of Urysohn spaces.

Our proof of Corollary~\ref{cor:urysohn} can be easily refined to recover indivisibility
of $S$-metric spaces with $S$ being a finite initial segment of integers.  
(To do that, the chains used to represent vertices in the proof of Corollary~\ref{cor:urysohn} have to be
chosen to all have the same embedding type. One can choose, for example, the embedding type created by the construction in the proof of Proposition~\ref{prop:universalposet} for the natural enumeration of this chain.)
It is also possible to adjust ideas from Section~\ref{sec:posets} to show finiteness of big Ramsey
degrees of many additional homogeneous structures resembling metric spaces~\cite{balko2021big}.
Furthermore, now that we have a proof technique showing that big Ramsey degrees are bounded for colouring of arbitrary
finite substructures of $S$-metric spaces, this naturally leads to generalizations of the concept of approximate
indivisibility to \emph{metric big Ramsey degrees}. This is developed in a follow-up paper~\cite{Bice2023}.

\subsection{Ultrametric spaces}
Recall that metric space $\str{M}=(M,d)$ is an \emph{ultrametric space} if the
triangle inequality can be strengthened to $d(u,w)\leq \max\{d(u,v),d(v,w)\}$.
The \emph{Urysohn ultrametric space of diameter $d$} is the universal and homogeneous ultrametric
space with distances $\{0,1,\ldots, d\}$.
The following was shown by Nguyen Van Th{\'e}~\cite{NVT2009} (along with a full characterisation
of big Ramsey degrees of ultrametric spaces):

\begin{theorem}
\label{cor:corollaryu}
For every $d\geq 1$ the Urysohn ultrametric space of diameter $d$ has finite big Ramsey degrees.
\end{theorem}
\begin{proof}
We construct an ultrametric space $\str{U}_d$ as follows:
\begin{enumerate}
 \item Vertices of $\str{U}_d$ are $d$-tuples of vertices of $\str{P}$.
 \item The distance between vertices $(u_0,u_1,\ldots u_{d-1})$ and $(v_0,v_1,\ldots, v_{d-1})$ is the minimal $\ell$ such that for every $0\leq i<d-\ell$ it holds that $u_i=v_i$.
\end{enumerate}
Again, it is easy to verify that this is a universal ultrametric space.  Finiteness of big Ramsey degrees now follows by an application of
Theorem~\ref{thm:posets}.
\end{proof}
Observe that one can replace $\str{P}$ by $\omega$ in the construction above and the same result (with better bounds)
follows by the infinite Ramsey theorem instead of Theorem~\ref{thm:posets}.
The construction above can be strengthened to all $\Lambda$-ultrametric spaces
where $\Lambda$ is a finite distributive lattice~\cite{sam}.

\subsection{Linear orders}
By fixing a linear extension of $\str{P}$ one obtains an alternative proof
of the Laver's result:
\begin{corollary}
The order of rationals has finite big Ramsey degrees.
\end{corollary}
While this may not be a very powerful observation on its own, we will discuss its
consequences in Corollary~\ref{cor:interposition}. Observe also that $\str{P}$
has a natural linear extension -- the lexicographic order.
\subsection{Structures with unary relations}
Another particularly simple consequence of Theorem~\ref{thm:posets} is:

\begin{corollary}
Let $L$ be a finite language consisting of unary relational symbols. Then the
universal homogeneous $L$-structure has finite big Ramsey degrees.
\end{corollary}
\begin{proof}
For simplicity assume that $L$ consists of a single unary relation $R$.
Then the universal $L$-structure can be represented using $\str{P}$ as follows:
\begin{enumerate}
\item Vertices are all pairs of distinct vertices of $\str{P}$.
\item Put vertex $(u_0,u_1)$ to the relation $R$ if and only if $u_0\leq_\str{P} u_1$.
\end{enumerate}
\end{proof}

\subsection{Free superpositions}
\label{superpose}
Recall that the \emph{age} of a structure $\str{M}$ is the set of all finite structures
having an embedding to $\str{M}$. Given a language $L$ and its sub-language $L^-\subseteq L$, 
an $L^-$-structure $\str{M}$ is the \emph{$L^-$-reduct} of an $L$-structure $\str{N}$
if $M=N$ and $\rel{M}{}=\rel{N}{}$ for every $\rel{}{}\in L^-$.

Let $L$ and $L'$ be languages such that $L\cap L'=\emptyset$. Let $\str{M}$ be
a homogeneous $L$-structure and $\str{N}$ a homogeneous $L'$-structure.
Then the \emph{free superposition of $\str{M}$ and $\str{N}$}, denoted by $\str{M}\ast\str{N}$,
is the homogeneous $L\cup L'$-structure
whose age consists precisely of those finite $(L\cup L')$-structures
with the property that their $L$-reduct is in the age of $\str{M}$
and $L'$-reduct is in the age of $\str{N}$~(see e.g. \cite{Bodirsky2015}).

It follows from the product Ramsey argument that the free interposition of
finitely many Ramsey classes with strong amalgamation property and no
algebraicity is also Ramsey~\cite[Lemma 3.22]{Bodirsky2015}, see also
\cite[Proposition 4.45]{Hubicka2016}. Similar general result is not known for big Ramsey structures.
However, we can combine the above observations to the following corollary (of Theorem~\ref{thm:posets2}) which heads in this direction by providing
means to interpose many of the known structures with finite big Ramsey degrees:

\begin{corollary}
\label{cor:interposition}
Let $\str{M}$ be a homogeneous structure that is a free superposition of finitely many copies of structures from the following list (each in a language disjoint from the others):
\begin{enumerate}
 \item the homogeneous universal partial order,
 \item the homogeneous universal triangle-free graph,
 \item the Urysohn $S$-metric space for a finite tight set $S$ (for $S=\{0,1,2\}$ one obtains the Rado graph),
 \item the Urysohn ultrametric space of a finite diameter $d$,
 \item the order of rationals,
 \item the homogeneous universal structure in a finite unary relational language,
\end{enumerate}
then $\str{M}$ has finite big Ramsey degrees.
\end{corollary}
\begin{proof}
Let $\str{M}_1,\str{M}_2,\ldots, \str{M}_n$ be structures from the statement of
the corollary, in mutually
disjoint languages $L_1,L_2,\ldots,L_n$ such that for every $1\leq i\leq n$ it holds that $\str{M}_i$ is $L_i$-structure.
Put $\str{M}=\str{M}_1\ast\str{M}_2\ast\cdots\ast\str{M}_n$.

As we showed above, for each structure $\str{M}_i$, $1\leq i\leq n$, there exists a
structure $\str{N}_i$ and an embedding $e_i\colon\str{M}_i\to\str{N}_i$
such that $N_i=\Emb(\str{V}_i,\str{P})$ for some finite structure $\str{V}_i$
and $\str N_i$ is represented using the partial order $\str{P}$ (or its linear extension).

Now consider a $(L_1\cup L_2\cup\cdots\cup L_n)$-structure $\str{N}$ defined as follows.  The vertex set $N$ of
$\str{N}$ consists of all $n$-tuples $(\vec{v}_1,\ldots, \vec{v}_n)$ with the
property that for every $1\leq i\leq n$ it holds that $\vec{v}_i$ is a vertex
of $\str{N}_i$.  Denote by $\pi_i$ the $i$-th projection
$(\vec{v}_1,\ldots, \vec{v}_n)\mapsto \vec{v}_i$.

Now we define relations of $\str{N}$. For every $1\leq i\leq n$, consider structure $\str{N}_i$:
\begin{enumerate}
 \item If $\str{N}_i$ is a partial order, then the corresponding partial order of $\str{N}$ is created by putting $u\leq v$ if and only if either $u=v$ or $\pi_i(u)\neq\pi_i(v)$ and $\pi_i(u)\leq_{\str{N}_i} \pi_i(v)$.
 \item If $\str{N}_i$ is homogeneous universal triangle-free graph then we put $u$ and $v$ adjacent if and only if $\pi_i(u)$ is adjacent to $\pi_i(v)$ in $\str{N}_i$.
 \item If $\str{N}_i$ is the order of rationals, then the corresponding linear order of $\str{N}$ is any linear order satisfying that $\pi_i$ is a monotone function.
 \item If $\str{N}_i$ is an $S$-metric space, then the corresponding metric space on $\str{N}$ is created by defining a distance of $u$ and $v$ to be $0$ if $u=v$, $\min (S\setminus \{0\})$ if $\pi_i(u)=\pi_i(v)$ and the distance of $\pi_i(u)$ and $\pi_i(v)$ otherwise.
 \item If $\str{N}_i$ is an ultrametric space then the corresponding ultrametric space is created analogously, but by putting the distance to be $1$ for every $u\neq v$, $\pi_i(u)=\pi_i(v)$.
 \item If $\str{N}_i$ is a structure with unary relations, then for every relation $R\in L_i$ we put $v$ to $R_\str{N}$ if and only if $\pi_i(v)\in R$.
\end{enumerate}
We say that substructure $\str{A}$ of $\str{N}$ is \emph{transversal} if for every two distinct vertices $(\vec{u}_1,\allowbreak \vec{u}_2,\allowbreak \ldots,\allowbreak \vec{u}_n),(\vec{v}_1,\allowbreak \vec{v}_2,\allowbreak \ldots,\allowbreak \vec{v}_n)\in A$ and every $1\leq i\leq n$ it holds that $\vec{u}_i\neq \vec{v}_i$.
Observe that embeddings $e_i\colon\str{M}_i\to\str{N}_i$, $1\leq i\leq n$, can be combined to an embedding $e\colon\str{M}\to\str{N}$ defined by putting $e(v)\mapsto (e_1(v),\allowbreak e_2(v),\allowbreak \ldots,\allowbreak e_n(v))$, and that the image $e(\str{M})$ is transversal.
One can also verify that for every $1\leq i\leq n$ it holds that the age of $\str{N}_i$ is the same as the age of the $L_i$-reduct of $\str N$. It follows that $\str N$ and $\str M$ have the same ages.
By universality of $\str M$ it follows that there is also an embedding $f\colon\str{N}\to\str{M}$.

Fix a finite structure $\str{A}$ and a finite colouring $\chi$ of $\Emb(\str{A},\str{M})$. Denote by $\mathcal A$ the set of all transversal structures in $\Emb(\str{A},\str{N})$.
Consider a finite colouring $\chi'$ of $\mathcal A$ defined by $\chi'(\widetilde{\str{A}})=\chi(f(\widetilde{\str{A}}))$.
For every $1\leq i\leq n$ this colouring projects by $\pi_i$ to a finite colouring of finite substructures of $\str{N}_i$ and consequently also of $\str{O}$.
This follows from the fact that the vertex set of $\str{N}_j$ is $\Emb(\str{N}_j,\str{O})$, for every $1\leq j\leq n$ and thus preimages of vertices in projection $\pi_i$ are all finite and isomorphic.
By a repeated application of Theorem~\ref{thm:posets2} it follows that $\str{N}$ has
finite big Ramsey degrees.  By the existence of embedding $e$, the corollary follows.
\end{proof}
Corollary \ref{cor:interposition} has further consequences. Superposing the Rado graph (which is the Urysohn $S$-metric space for $S=\{0,1,2\}$)
and the universal homogeneous structure in the language with one unary relation one can obtain that the random countable bipartite graph has finite big Ramsey degrees.
This follows from the fact that the random countable bipartite graph can be defined in the superposition by considering only those edges
where precisely one of the endpoints is in the unary relation.

Similarly, by superposing the linear order with the universal homogeneous structure in the language with one unary relation it follows that the homogeneous dense local order
has finite big Ramsey degrees (this was shown by Laflamme, Nguyen Van Thé, and Sauer~\cite{laflamme2010partition}).
Superposing multiple linear and partial orders leads to big Ramsey equivalents of results of Soki\'c~\cite{sokic2013ramsey}, Solecki and Zhao~\cite{solecki2017ramsey},
and Dragani{\'c} and Ma{\v{s}}ulovi{\'c}~\cite{draganic2019ramsey}.

\section{Concluding remarks}\label{sec:concluding}
\subsection{Bigger forbidden substructures and bigger arities}
The method presented in this paper can be used to strengthen 
Theorem~\ref{thm:trianglefree} for free amalgamation classes in finite binary languages defined by finitely many forbidden 
irreducible substructures on at most $3$ vertices.

For non-binary relations and bigger forbidden irreducible substructures, it seems
necessary to refine Theorem~\ref{thm:multCS} for colouring multi-dimensional
objects rather than words in a similar manner as in~\cite{Hubickabigramsey,Hubicka2020uniform}
and is presently a work in progress~\cite{Balko2023Sucessor}.
This seems to further develop the link between constructions in the structural Ramsey theory
and the extension property for partial automorphisms~\cite{Hubicka2018EPPA}.
\subsection{Optimality}
\label{sec:optimality}
The big Ramsey degree of a vertex in the universal homogeneous tri\-an\-gle-free graph was
shown to be one by Komj\'ath and R\"odl~\cite{komjath1986} in 1986. The big Ramsey degree of an edge is four as shown by Sauer~\cite{sauer1998} in 1998.
Proofs of Theorems~\ref{thm:posets} and \ref{thm:trianglefree} can be refined to
exactly describe the big Ramsey degrees similarly as was done by 
Sauer~\cite{Sauer2006} for the random graph and Laflamme, Sauer, and Vuksanovic for free binary structures~\cite{Laflamme2006}. This leads
to big Ramsey structures as defined by Zucker~\cite{zucker2017}.  
Work on exact big Ramsey degrees of triangle-free graph based on the concepts of Section~\ref{sec:trianglefree}
eventually led to a more general result exactly characterising big Ramsey degrees of free amalgamation classes
in finite binary languages with finitely many forbidden substructures~\cite{Balko2021exact}.

Work on exact big Ramsey degrees of the universal homogeneous partial order led to interesting refinements
of the underlying Ramsey theorem and will appear in a follow-up paper~\cite{Balko2023}.

\subsection*{Acknowledgments}
I am grateful to the anonymous referee, Martin Balko, David Cho\-douns\-k\'y, Natasha Dobrinen, Mat\v ej Kone\v cn\'y, Jaroslav Ne\v set\v ril, Stevo Todorcevic, Lluis Vena and Andy Zucker for remarks and
corrections which improved presentation of this paper.  This work was directly
motivated by discussions with Jaroslav Ne\v set\v ril and also with Martin Balko, David Chodounsk\'y, Mat\v ej Kone\v
cn\'y, and Lluis Vena during the meetings of project 18--13685Y of the  Czech  Science Foundation (GA\v CR).  
I was introduced to this interesting area by
Claude Laflamme, Norbert Sauer, and Robert Woodrow during my post-doc stay in Calgary.

\bibliographystyle{plain}

\bibliography{ramsey.bib}

\begin{thebibliography}{10}

\bibitem{Balko2023}
Martin Balko, David Chodounsk{\' y}, Natasha Dobrinen, Jan Hubi{\v c}ka, Mat{\v
  e}j Kone{\v c}n{\' y}, Lluis Vena, and Andy Zucker.
\newblock Characterisation of the big {R}amsey degrees of the generic partial
  order.
\newblock Submitted, arXiv:2303.10088, 2023.

\bibitem{Balko2021exact}
Martin Balko, David Chodounsk{\' y}, Natasha Dobrinen, Jan Hubi{\v c}ka, Mat{\v
  e}j Kone{\v c}n{\' y}, Lluis Vena, and Andy Zucker.
\newblock Exact big {R}amsey degrees for finitely constrained binary free
  amalgamation classes.
\newblock {\em Journal of the European Mathematical Society}, August 2024.

\bibitem{Balko2023Sucessor}
Martin Balko, David Chodounsk{\' y}, Natasha Dobrinen, Jan Hubi{\v c}ka,
  Jaroslav Ne{\v{s}}et{\v{r}}il, Mat{\v e}j Kone{\v c}n{\' y}, and Andy Zucker.
\newblock Ramsey theorem for trees with successor operation.
\newblock Preprint, arXiv:2311.06872, 2023.

\bibitem{balko2021big}
Martin Balko, David Chodounsk{\'y}, Jan Hubi{\v{c}}ka, Mat{\v{e}}j
  Kone{\v{c}}n{\'y}, Jaroslav Ne{\v{s}}et{\v{r}}il, and Lluis Vena.
\newblock Big {R}amsey degrees and forbidden cycles.
\newblock In Jaroslav Ne{\v{s}}et{\v{r}}il, Guillem Perarnau, Juanjo Ru{\'e},
  and Oriol Serra, editors, {\em Extended Abstracts EuroComb 2021}, pages
  436--441. Springer International Publishing, 2021.

\bibitem{Hubickabigramsey}
Martin Balko, David Chodounsk{\' y}, Jan Hubi{\v c}ka, Mat{\v e}j Kone{\v
  c}n{\' y}, and Lluis Vena.
\newblock Big {R}amsey degrees of 3-uniform hypergraphs.
\newblock {\em Acta Mathematica Universitatis Comenianae}, 88(3):415--422,
  2019.
\newblock Extended abstract for Eurocomb 2019.

\bibitem{Hubicka2020uniform}
Martin Balko, David Chodounsk{\'y}, Jan Hubi{\v{c}}ka, Mat{\v{e}}j
  Kone{\v{c}}n{\'y}, and Lluis Vena.
\newblock Big {R}amsey degrees of 3-uniform hypergraphs are finite.
\newblock {\em Combinatorica}, 42(2):659--672, 2022.

\bibitem{Bice2023}
Tristan Bice, Noe de~Racourt, Jan Hubi{\v c}ka, and Mat{\v e}j Kone{\v c}n{\'
  y}.
\newblock Big {R}amsey degrees in the metric setting.
\newblock In Daniel Kráľ and Jaroslav Ne{\v{s}}et{\v{r}}il, editors, {\em
  Proceedings of the 12th European Conference on Combinatorics, Graph Theory
  and Applications EUROCOMB’23}, pages 134--141. MUNI Press, 2023.

\bibitem{Bodirsky2015}
Manuel Bodirsky.
\newblock Ramsey classes: Examples and constructions.
\newblock {\em Surveys in Combinatorics 2015}, 424:1, 2015.

\bibitem{sam}
Samuel Braunfeld.
\newblock Ramsey expansions of {$\Lambda$}-ultrametric spaces.
\newblock Accepted to \emph{European Journal of Combinatorics},
  arXiv:1710.01193, 2017.

\bibitem{braunfeld2023big}
Samuel Braunfeld, David Chodounsk{\'y}, No{\'e} de~Rancourt, Jan Hubi{\v{c}}ka,
  Jamal Kawach, and Mat{\v{e}}j Kone{\v{c}}n{\'y}.
\newblock Big {R}amsey {D}egrees and {I}nfinite {L}anguages.
\newblock {\em Advances in Combinatorics}, 2024:4, 2024.
\newblock 26pp.

\bibitem{cameron1997random}
Peter~J Cameron.
\newblock The random graph.
\newblock {\em The Mathematics of Paul Erd{\"o}s II}, pages 333--351, 1997.

\bibitem{carlson1987infinitary}
Timothy~J. Carlson.
\newblock An infinitary version of the {G}raham--{L}eeb--{R}othschild theorem.
\newblock {\em Journal of Combinatorial Theory, Series A}, 44(1):22--33, 1987.

\bibitem{carlson1984}
Timothy~J. Carlson and Stephen~G. Simpson.
\newblock A dual form of {R}amsey's theorem.
\newblock {\em Advances in Mathematics}, 53(3):265--290, 1984.

\bibitem{Delhomme2007}
Christian Delhomm{\'e}, Claude Laflamme, Maurice Pouzet, and Norbert~W. Sauer.
\newblock Divisibility of countable metric spaces.
\newblock {\em European Journal of Combinatorics}, 28(6):1746--1769, 2007.

\bibitem{delhomme2008indivisible}
Christian Delhomm{\'e}, Claude Laflamme, Maurice Pouzet, and Norbert~W. Sauer.
\newblock Indivisible ultrametric spaces.
\newblock {\em Topology and its Applications}, 155(14):1462--1478, 2008.

\bibitem{Dellamonica2012}
Domingos Dellamonica and Vojt{\v{e}}ch R{\"o}dl.
\newblock Distance preserving {R}amsey graphs.
\newblock {\em Combinatorics, Probability and Computing}, 21(04):554--581,
  2012.

\bibitem{devlin1979}
Denis Devlin.
\newblock {\em Some partition theorems and ultrafilters on $\omega$}.
\newblock PhD thesis, Dartmouth College, 1979.

\bibitem{dobrinen2017universal}
Natasha Dobrinen.
\newblock The {R}amsey theory of the universal homogeneous triangle-free graph.
\newblock {\em Journal of Mathematical Logic}, 20(02):2050012, 2020.

\bibitem{dobrinen2019ramsey}
Natasha Dobrinen.
\newblock The {R}amsey theory of {H}enson graphs.
\newblock {\em Journal of Mathematical Logic}, 23(01):2250018, 2023.

\bibitem{dobrinen2016rainbow}
Natasha Dobrinen, Claude Laflamme, and Norbert Sauer.
\newblock Rainbow {R}amsey simple structures.
\newblock {\em Discrete Mathematics}, 339(11):2848--2855, 2016.

\bibitem{dodos2016}
Pandelis Dodos and Vassilis Kanellopoulos.
\newblock {\em Ramsey theory for product spaces}, volume 212.
\newblock American Mathematical Society, 2016.

\bibitem{dodos2014}
Pandelis Dodos, Vassilis Kanellopoulos, and Konstantinos Tyros.
\newblock Measurable events indexed by words.
\newblock {\em Journal of Combinatorial Theory, Series A}, 127:176--223, 2014.

\bibitem{draganic2019ramsey}
Nemanja Dragani{\'c} and Dragan Ma{\v{s}}ulovi{\'c}.
\newblock A {R}amsey theorem for multiposets.
\newblock {\em European Journal of Combinatorics}, 81:142--149, 2019.

\bibitem{erdos1974unsolved}
Paul Erd{\H{o}}s and Andr{\'a}s Hajnal.
\newblock Unsolved and solved problems in set theory.
\newblock In {\em Proceedings of the Tarski Symposium (Berkeley, Calif., 1971),
  Amer. Math. Soc., Providence}, volume~1, pages 269--287, 1974.

\bibitem{Fouche1997}
Willem~L Fouch{\'e}.
\newblock Symmetry and the {R}amsey degree of posets.
\newblock {\em Discrete Mathematics}, 167:309--315, 1997.

\bibitem{furstenberg1989}
Hilel F{\"u}rstenberg and Yitzhak Katznelson.
\newblock Idempotents in compact semigroups and {R}amsey theory.
\newblock {\em Israel Journal of Mathematics}, 68(3):257--270, 1989.

\bibitem{Graham1971}
Ronald~L. Graham and Bruce~L. Rothschild.
\newblock {R}amsey's theorem for $n$-parameter sets.
\newblock {\em Transactions of the American Mathematical Society},
  159:257--292, 1971.

\bibitem{Halpern1966}
James~D. Halpern and Hans L{\"a}uchli.
\newblock A partition theorem.
\newblock {\em Transactions of the American Mathematical Society},
  124(2):360--367, 1966.

\bibitem{hedrlin1969universal}
Zden{\v{e}}k Hedrl{\'\i}n.
\newblock On universal partly ordered sets and classes.
\newblock {\em Journal of Algebra}, 11(4):503--509, 1969.

\bibitem{HubickaKonecnySurvey}
Jan Hubi{\v c}ka and Mat{\v e}j Kone{\v c}n{\'y}.
\newblock Twenty years of {N}ešetřil’s classification programme of {R}amsey
  classes.
\newblock Submitted, 2025.

\bibitem{Hubicka2018EPPA}
Jan Hubi{\v{c}}ka, Mat{\v{e}}j Kone{\v{c}}n{\'{y}}, and Jaroslav
  Ne{\v{s}}et{\v{r}}il.
\newblock All those {E}{P}{P}{A} classes (strengthenings of the
  {H}erwig--{L}ascar theorem).
\newblock {\em Transactions of the American Mathematical Society},
  375(11):7601--7667, 2022.

\bibitem{Hubicka2005a}
Jan Hubi{\v{c}}ka and Jaroslav Ne{\v{s}}et{\v{r}}il.
\newblock Finite presentation of homogeneous graphs, posets and {R}amsey
  classes.
\newblock {\em Israel Journal of Mathematics}, 149(1):21--44, 2005.

\bibitem{hubicka2011some}
Jan Hubi{\v {c}}ka and Jaroslav Ne{\v{s}}etril.
\newblock Some examples of universal and generic partial orders.
\newblock {\em Model Theoretic Methods in Finite Combinatorics}, pages
  293--318, 2011.

\bibitem{Hubicka2016}
Jan Hubi{\v{c}}ka and Jaroslav Ne\v{s}et\v{r}il.
\newblock All those {R}amsey classes ({R}amsey classes with closures and
  forbidden homomorphisms).
\newblock {\em Advances in Mathematics}, 356C:106791, 2019.

\bibitem{Hubicka2005}
Jan Hubi\v{c}ka and Jaroslav Ne\v{s}et\v{r}il.
\newblock Universal partial order represented by means of oriented trees and
  other simple graphs.
\newblock {\em European Journal of Combinatorics}, 26(5):765--778, 2005.

\bibitem{Karagiannis2013}
Nikolaos Karagiannis.
\newblock A combinatorial proof of an infinite version of the {H}ales--{J}ewett
  theorem.
\newblock {\em Journal of Combinatorics}, 4(2):273--291, 2013.

\bibitem{Kechris2005}
Alexander~S. Kechris, Vladimir~G. Pestov, and Stevo Todor{\v c}evi{\' c}.
\newblock Fra{\"\i}ss{\'e} limits, {R}amsey theory, and topological dynamics of
  automorphism groups.
\newblock {\em Geometric and Functional Analysis}, 15(1):106--189, 2005.

\bibitem{komjath1986}
P{\'e}ter Komj{\'a}th and Vojtech R{\"o}dl.
\newblock Coloring of universal graphs.
\newblock {\em Graphs and Combinatorics}, 2(1):55--60, 1986.

\bibitem{Konecny2018b}
Mat{\v e}j Kone{\v c}n{\'y}.
\newblock Semigroup-valued metric spaces.
\newblock Master's thesis, Charles University, 2019.
\newblock arXiv:1810.08963.

\bibitem{laflamme2010partition}
Claude Laflamme, Lionel Nguyen Van~Thé, and Norbert~W. Sauer.
\newblock Partition properties of the dense local order and a colored version
  of {M}illiken’s theorem.
\newblock {\em Combinatorica}, 30(1):83--104, 2010.

\bibitem{Laflamme2006}
Claude Laflamme, Norbert~W. Sauer, and Vojkan Vuksanovic.
\newblock Canonical partitions of universal structures.
\newblock {\em Combinatorica}, 26(2):183--205, 2006.

\bibitem{larson2008counting}
Jean~A. Larson.
\newblock Counting canonical partitions in the random graph.
\newblock {\em Combinatorica}, 28(6):659--678, 2008.

\bibitem{laver1984products}
Richard Laver.
\newblock Products of infinitely many perfect trees.
\newblock {\em Journal of the London Mathematical Society}, 2(3):385--396,
  1984.

\bibitem{lopez2008oscillation}
Jordi Lopez-Abad and Lionel Nguyen Van~Th{\'e}.
\newblock The oscillation stability problem for the urysohn sphere: A
  combinatorial approach.
\newblock {\em Topology and its Applications}, 155(14):1516--1530, 2008.

\bibitem{Macpherson2011}
Dugald Macpherson.
\newblock A survey of homogeneous structures.
\newblock {\em Discrete Mathematics}, 311(15):1599--1634, 2011.
\newblock Infinite Graphs: Introductions, Connections, Surveys.

\bibitem{masulovic2016pre}
Dragan Ma{\v{s}}ulovi{\'c}.
\newblock Pre-adjunctions and the {R}amsey property.
\newblock {\em European Journal of Combinatorics}, 70:268--283, 2018.

\bibitem{Milliken1979}
Keith~R. Milliken.
\newblock A {R}amsey theorem for trees.
\newblock {\em Journal of Combinatorial Theory, Series A}, 26(3):215--237,
  1979.

\bibitem{milman1971new}
Vitali~D. Milman.
\newblock A new proof of a {D}voretzky's theorem on cross-sections of convex
  bodies.
\newblock {\em Funkcional Anal. i Prilozen}, 5:28--37, 1971.

\bibitem{milman1992dvoretzky}
Vitali~D. Milman.
\newblock Dvoretzky's theorem—thirty years later.
\newblock {\em Geometric \& Functional Analysis GAFA}, 2(4):455--479, 1992.

\bibitem{Nevsetvril2007}
Jaroslav Ne{\v{s}}et{\v{r}}il.
\newblock Metric spaces are {R}amsey.
\newblock {\em European Journal of Combinatorics}, 28(1):457--468, 2007.

\bibitem{nevsetril1975ramsey}
Jaroslav Ne{\v{s}}etril and Vojt{\v{e}}ch R{\"o}dl.
\newblock A {R}amsey graph without triangles exists for any graph without
  triangles.
\newblock {\em Infinite and finite sets III}, 10:1127--1132, 1975.

\bibitem{nesetril1975type}
Jaroslav Ne{\v{s}}et{\v{r}}il and Vojt{\v{e}}ch R{\"o}dl.
\newblock Type theory of partition properties of graphs.
\newblock In {\em Recent advances in graph theory}, pages 405--412. Academia
  Prague, 1975.

\bibitem{Nevsetvril1977}
Jaroslav Ne{\v{s}}et{\v{r}}il and Vojt{\v{e}}ch R{\"o}dl.
\newblock A structural generalization of the {R}amsey theorem.
\newblock {\em Bulletin of the American Mathematical Society}, 83(1):127--128,
  1977.

\bibitem{Nevsetvril1984}
Jaroslav Ne{\v{s}}et{\v{r}}il and Vojt{\v{e}}ch R{\"o}dl.
\newblock Combinatorial partitions of finite posets and lattices---{R}amsey
  lattices.
\newblock {\em Algebra Universalis}, 19(1):106--119, 1984.

\bibitem{Nevsetvril1989}
Jaroslav Ne{\v{s}}et{\v{r}}il and Vojt{\v{e}}ch R{\"o}dl.
\newblock The partite construction and {R}amsey set systems.
\newblock {\em Discrete Mathematics}, 75(1):327--334, 1989.

\bibitem{nevsetvril2018ramsey}
Jaroslav Ne{\v{s}}et{\v{r}}il and Vojt{\v{e}}ch R{\"o}dl.
\newblock Ramsey partial orders from acyclic graphs.
\newblock {\em Order}, 35(2):293--300, 2018.

\bibitem{NVT2009}
Lionel Nguyen Van~Th{\'e}.
\newblock Ramsey degrees of finite ultrametric spaces, ultrametric {U}rysohn
  spaces and dynamics of their isometry groups.
\newblock {\em European Journal of Combinatorics}, 30(4):934--945, 2009.

\bibitem{The2010}
Lionel Nguyen Van~Th{\'e}.
\newblock {\em Structural {R}amsey Theory of Metric Spaces and Topological
  Dynamics of Isometry Groups}.
\newblock Memoirs of the American Mathematical Society. American Mathematical
  Society, 2010.

\bibitem{NVT2009b}
Lionel Nguyen Van~Th{\'e} and Norbert~W. Sauer.
\newblock The {U}rysohn sphere is oscillation stable.
\newblock {\em Geometric and Functional Analysis}, 19(2):536--557, Sep 2009.

\bibitem{Trotter1985}
Madeleine Paoli, William~T. Trotter, and James~W. Walker.
\newblock Graphs and orders in {R}amsey theory and in dimension theory.
\newblock In Ivan Rival, editor, {\em Graphs and Order}, volume 147 of {\em
  NATO AST series}, pages 351--394. Springer, 1985.

\bibitem{PromelBook}
Hans~J\"urgen Pr{\"o}mel.
\newblock {\em Ramsey {T}heory for {D}iscrete {S}tructures}.
\newblock Springer International Publishing, 2013.

\bibitem{promel1985baire}
Hans~J{\"u}rgen Pr{\"o}mel and Bernd Voigt.
\newblock Baire sets of $k$-parameter words are {R}amsey.
\newblock {\em Transactions of the American Mathematical Society},
  291(1):189--201, 1985.

\bibitem{pultr1980combinatorial}
Ale{\v s} Pultr and V{\v e}ra Trnkov{\'a}.
\newblock {\em Combinatorial, Algebraic, and Topological Representations of
  Groups, Semigroups, and Categories}.
\newblock Mathematical Studies. North-Holland Publishing Company, 1980.

\bibitem{sauer1998}
Norbert~W. Sauer.
\newblock Edge partitions of the countable triangle free homogeneous graph.
\newblock {\em Discrete Mathematics}, 185(1-3):137--181, 1998.

\bibitem{Sauer2006}
Norbert~W. Sauer.
\newblock Coloring subgraphs of the {R}ado graph.
\newblock {\em Combinatorica}, 26(2):231--253, 2006.

\bibitem{Sauer2012}
Norbert~W. Sauer.
\newblock Vertex partitions of metric spaces with finite distance sets.
\newblock {\em Discrete Mathematics}, 312(1):119--128, 2012.

\bibitem{Sauer2013}
Norbert~W. Sauer.
\newblock Distance sets of {U}rysohn metric spaces.
\newblock {\em Canadian Journal of Mathematics}, 65(1):222--240, 2013.

\bibitem{Sauer2014}
Norbert~W. Sauer.
\newblock Age and weak indivisibility.
\newblock {\em European Journal of Combinatorics}, 37:24--31, 2014.

\bibitem{sokic2012ramsey}
Miodrag Soki{\'c}.
\newblock Ramsey properties of finite posets.
\newblock {\em Order}, 29(1):1--30, 2012.

\bibitem{sokic2013ramsey}
Miodrag Soki{\'c}.
\newblock Ramsey property, ultrametric spaces, finite posets, and universal
  minimal flows.
\newblock {\em Israel Journal of Mathematics}, 194(2):609--640, 2013.

\bibitem{solecki2017ramsey}
S{\l}awomir Solecki and Min Zhao.
\newblock A {R}amsey theorem for partial orders with linear extensions.
\newblock {\em European Journal of Combinatorics}, 60:21--30, 2017.

\bibitem{todorcevic2010introduction}
Stevo Todorcevic.
\newblock {\em Introduction to {R}amsey spaces}, volume 174.
\newblock Princeton University Press, 2010.

\bibitem{zucker2017}
Andy Zucker.
\newblock Big {R}amsey degrees and topological dynamics.
\newblock {\em Groups, Geometry, and Dynamics}, 13(1):235--276, 2019.

\bibitem{zucker2020}
Andy Zucker.
\newblock On big {R}amsey degrees for binary free amalgamation classes.
\newblock {\em Advances in Mathematics}, 408:108585, 2022.

\end{thebibliography}
\end{document}